\documentclass[12pt]{article}

\usepackage{epsfig,amsmath,amsfonts,amssymb,latexsym,color,amsthm,hhline,multirow,subfigure,graphicx,bm}
\usepackage{enumerate}

\usepackage{tikz}
\usetikzlibrary{hobby}

\newtheorem{theorem}{Theorem}
\newtheorem{conjecture}[theorem]{Conjecture}
\newtheorem{prop}[theorem]{Proposition}
\newtheorem{lemma}[theorem]{Lemma}

\newcommand{\E}{{\mathbb E}}
\newcommand{\RR}{{\mathbb R}}

\newcommand{\Z}{{\mathbb Z}}
\newcommand {\PP}{{\mathbb P}}
\newcommand{\sss}{\scriptscriptstyle}

\newcommand{\N}{\mathbb{N}}

\newcommand{\vep}{\varepsilon}

\newcommand{\Dt}{D_e^xT_t}
\newcommand{\Do}{D_e^xT_0}

\newcommand{\Ec}{\mathcal{E}}

\newcommand{\Inf}{{\textup{Inf}}}
\newcommand{\Var}{{\textup{Var}}}
\newcommand{\Cov}{{\textup{Cov}}}
\newcommand{\Corr}{{\textup{Corr}}}

\begin{document}

\title{Chaos, concentration and multiple valleys\\ in first-passage percolation}\parskip=5pt plus1pt minus1pt \parindent=0pt
\author{Daniel Ahlberg\thanks{Department of Mathematics, Stockholm University.\newline\hspace*{0.5cm}{\tt \{daniel.ahlberg\}\{mia\}\{matteo.sfragara\}@math.su.se}}\and Maria Deijfen\footnotemark[1] \and Matteo Sfragara\footnotemark[1]}
\date{October 2024}
\maketitle

\begin{abstract}
\noindent A decade and a half ago Chatterjee established the first rigorous connection between anomalous fluctuations and a chaotic behaviour of the ground state in certain Gaussian disordered systems. The purpose of this paper is to show that Chatterjee's work gives evidence of a more general principle, by establishing an analogous connection between fluctuations and chaos in the context of first-passage percolation. The notion of `chaos' here refers to the sensitivity of the time-minimising path between two points when exposed to a slight perturbation. More precisely, we resample a small proportion of the edge weights, and find that a vanishing fraction of the edges on the time-minimising path still belongs to the time-minimising path obtained after resampling. We also identify the point at which the system transitions from being stable to being chaotic in terms of the variance of the system. Finally we show that the chaotic behaviour implies the existence of a large number of almost-optimal paths that are almost disjoint from the time-minimising path, a phenomenon known as `multiple valleys'.

\vspace{0.3cm}


\vspace{0.2cm}

\end{abstract}

\section{Introduction}\label{sec:intro}

It has been known for quite some time that many interesting functions involving a large number of independent variables, each having little influence on the outcome as a whole, exhibit fluctuations at a lower order than what classical techniques would prescribe. This phenomenon has come to be referred to as \emph{anomalous fluctuations} or \emph{superconcentration}. While a robust theory for the concentration of measure has come to explain this phenomenon, its potential for improvement upon classical techniques is inhibited by its generality; see e.g.~\cite{boulugmas13,chatterjee14a}. Understanding the mechanism that dictates the order of fluctuations remains one of the most important open problems in stochastic models for spatial growth, such as first-passage percolation.

The study of how percolative systems are affected when exposed to noise or dynamics offers a rich structure and a plethora of interesting phenomena; see e.g.~\cite{garste15}. A fascinating connection between sensitivity to noise and fluctuations in stochastic growth was discovered in the seminal work of Benjamini, Kalai and Schramm~\cite{benkalsch99,benkalsch03}, via the study of influences. In two later preprints, Chatterjee~\cite{chatterjee1,chatterjee2} established a connection between superconcentration and a chaotic behaviour of the ground state in certain Gaussian disordered systems. (The two preprints were later combined and published as a book~\cite{chatterjee14a}, which we henceforth refer to.) The purpose of this paper is to show that Chatterjee's discovery gives evidence of a much more general principle, by establishing an analogous connection between fluctuations and chaos in the context of first-passage percolation.

Since the distance-metric in first-passage percolation is known to be superconcentrated, we are able to provide the first evidence of chaos in this setting. Our results illustrate once again that various aspects of first-passage percolation cannot be fully understood without also understanding the behaviour of distance-minimising paths (i.e.\ geodesics) of the metric.

\subsection{Model description}

First-passage percolation was introduced in the 1960s by Hammersley and Welsh~\cite{hamwel65}, and has since become an archetype among stochastic models for spatial growth. The model is constructed from a graph, typically the $\Z^d$ nearest neighbour lattice, for some $d\ge2$, and i.i.d.\ non-negative weights $\omega(e)$ assigned to the edges $e$ of that graph. The weights may be interpreted as traversal times for some infectious phenomena, and induce a random metric $T$ on the vertex set of the graph, defined as follows:\footnote{If weights may attain the value zero with positive probability, then $T$ is a pseudo-metric.} For $u,v\in\Z^d$ let
\begin{equation}\label{eq:firstTdef}
T(u,v):=\inf\bigg\{\sum_{e\in\Gamma}\omega(e):\Gamma\text{ is a nearest-neighbour path from $u$ to }v\bigg\}.
\end{equation}
Any nearest-neighbour path $\pi$ connecting $u$ to $v$ whose weight sum attains the infimum in~\eqref{eq:firstTdef} is referred to as a \emph{geodesic} between $u$ and $v$.

The occurrence of chaotic phenomena in models of disordered systems such as spin glasses was predicted by physicists in the 1980s~\cite{bramoo87,fishus86}. In the context of first-passage percolation, the analogous notion of `chaos' refers to the sensitivity of geodesics as the first-passage metric is exposed to small perturbations. In order to address this behaviour properly, we shall be interested in a dynamical version of first-passage percolation, where edges update their weight over time, according to independent uniform clocks. We let $\Ec$ denote the set of nearest-neighbour edges of the $\Z^d$ lattice, where $d\ge2$ is fixed. Let $\{\omega(e)\}_{e\in\Ec}$ and $\{\omega'(e)\}_{e\in\Ec}$ be two independent i.i.d.\ families of non-negative edge weights with common distribution $F$, and let $\{U(e)\}_{e\in\Ec}$ be an independent collection of independent variables uniformly distributed on the interval $[0,1]$. We define a \emph{dynamical} weight configuration $\omega_t=\{\omega_t(e)\}_{e\in\Ec}$ for $t\in[0,1]$ as
\begin{equation}\label{passage_times}
\omega_t(e):=\left\{
\begin{array}{ll}
\omega(e) & \mbox{if }U(e)>t,\\
\omega'(e) & \mbox{if }U(e)\leq t.
  \end{array}
            \right.
\end{equation}
In addition, we define the first-passage time between $u,v\in\Z^d$ at time $t\in[0,1]$ as
\begin{equation}\label{eq:Tdef}
T_t(u,v):=\inf\bigg\{\sum_{e\in\Gamma}\omega_t(e):\Gamma\text{ is a nearest-neighbour path from $u$ to }v\bigg\}.
\end{equation}

For fixed $t\in[0,1]$ the weight configuration $\omega_t$ consists of independent variables distributed as $F$, and hence $T_t(u,v)$ is equidistributed for all $t\in[0,1]$. We shall henceforth assume that $F$ has finite second moment and does not put too much mass at zero-weight edges, although some of our results hold under weaker assumptions. More specifically, we shall assume
\begin{equation}\label{eq:pc_moment}
F(0)<p_c(d)\quad\text{and}\quad\int x^2\,dF(x)<\infty,
\end{equation}
where $p_c(d)$ denotes the critical value for bond percolation on $\Z^d$. The former condition is mainly present in order to guarantee that the infimum in~\eqref{eq:Tdef} is attained for some path $\Gamma$, so that a geodesic between $u$ and $v$ exists almost surely; see e.g.~\cite[Section~4.1]{aufdamhan17}.

When $F$ is continuous, there exists an almost surely unique geodesic between any two points. In this case we denote by $\pi_t(u,v)$ the geodesic between $u$ and $v$ at time $t\in[0,1]$, i.e.\ the path attaining the infimum in~\eqref{eq:Tdef}. When $F$ has atoms, there may be multiple paths attaining the infimum in~\eqref{eq:Tdef}. In this case we define $\pi_t(u,v)$ as the intersection of all geodesics between $u$ and $v$, i.e.\
\begin{equation}\label{eq:pi_def}
\pi_t(u,v):=\big\{e\in\Ec:e\in\Gamma\text{ for every $\Gamma$ that attains the infimum in~\eqref{eq:Tdef}}\big\}.
\end{equation}
Of course, when the geodesic between $u$ and $v$ is unique, then $\pi_t(u,v)$ denotes this geodesic, so the latter definition of $\pi_t(u,v)$ is indeed an extension of the former. In either case, assuming only~\eqref{eq:pc_moment}, it is a known fact that there exists $c>0$ such that
\begin{equation}\label{eq:linear}
c|v|\le\E[|\pi_t(0,v)|]\le \frac{1}{c}|v|\quad\text{for all }v\in\Z^d;
\end{equation}
see e.g.~\cite[Section~4.1]{aufdamhan17}.

A notoriously challenging problem in first-passage percolation, and more generally for models of spatial growth and disordered systems, is to determine the order of fluctuations of the random metric $T$. In two dimensions, first-passage percolation is believed to pertain to the KPZ-class of universality, which would suggest that $T(0,v)$ should typically fluctuate at order $|v|^{1/3}$ around its mean. In higher dimensions the fluctuations are believed to be smaller still. Classical techniques, such as an Efron-Stein estimate, give a linear upper bound on the variance of the first-passage metric. The best current bounds provide a modest improvement upon the classical techniques, and yield that for every $d\ge2$ there exists a constant $C<\infty$ such that, for all $v\in\Z^d$, we have
\begin{equation}\label{eq:sublin}
\Var\big(T(0,v)\big)\le C\frac{|v|}{\log |v|}.
\end{equation}
The bound in~\eqref{eq:sublin} was first obtained by Benjamini, Kalai and Schramm~\cite{benkalsch03} in the special case of $\{a,b\}$-valued edge weights, for $0<a<b<\infty$. Their result was later extended to a larger class of weight distributions by Bena\"im and Rossignol~\cite{benros,benros08} and Damron, Hanson and Sosoe~\cite{damhansos15}, and is now known to hold for weight distributions satisfying
\begin{equation}\label{eq:var_cond}
F(0)<p_c(d)\quad\text{and}\quad\int x^2(\log x)_+\,dF(x)<\infty.
\end{equation}
Polynomial improvement on the bound in~\eqref{eq:sublin} is widely recognised as one of the most central open problems in the study of first-passage percolation.

\subsection{Main results}

Chatterjee's work on disordered Gaussian systems~\cite{chatterjee14a} suggests that there is a connection between fluctuations of $T(0,v)$ and the dynamical behaviour of the geodesic $\pi_t(0,v)$ through the relation\footnote{Here, for sequences $(a_v)$ and $(b_v)$ of positive real numbers indexed by $\Z^d$, the relation $a_v\asymp b_v$ denotes that there exist constants $0<C_1,C_2<\infty$ such that $C_1b_v\leq a_v\leq C_2b_v$ for all $v$.}
\begin{equation}\label{eq:fluct_geos}
\Var\big(T(0,v)\big)\asymp\int_0^1\E\big[\big|\pi_0(0,v)\cap\pi_t(0,v)\big|\big]\,dt.
\end{equation}
A standard coupling argument shows that the expected overlap of the two paths in the right-hand side of~\eqref{eq:fluct_geos} is decreasing as a function of $t$ (see Lemma \ref{lma:overlap_monotonicity}). Consequently, it is a straightforward consequence from~\eqref{eq:fluct_geos} that sublinear variance growth $\Var(T(0,v))=o(|v|)$ is equivalent to sublinear growth of the overlap $\E[|\pi_0(0,v)\cap\pi_t(0,v)|]=o(|v|)$ for $t>0$ fixed. In view of the linear growth of geodesics~\eqref{eq:linear}, the latter is an expression of the chaotic nature of the first-passage metric as referred to above. In combination with the variance bound~\eqref{eq:sublin}, it will suffice to establish~\eqref{eq:fluct_geos} in order to deduce the chaotic behaviour of the first-passage metric. This is the essence of Chatterjee's work for certain disordered Gaussian systems~\cite{chatterjee14a}.

Extrapolating from~\eqref{eq:fluct_geos}, one is led to believe in the slightly more general formula (the two coincide for $t=1$) stating that\footnote{Here we say that the relation $a_v(t)\asymp b_v(t)$ holds if there exist constants $0<C_1,C_2<\infty$ such that $C_1b_v(t)\leq a_v(t)\leq C_2b_v(t)$ for all $v\in\Z^d$ and $t\in[0,1]$.}
\begin{equation}\label{eq:cov_geos}
\E\Big[\big(T_0(0,v)-T_t(0,v)\big)^2\Big]\asymp\int_0^t\E\big[\big|\pi_0(0,v)\cap\pi_s(0,v)\big|\big]\,ds.
\end{equation}
Again working in a Gaussian setting, Ganguly and Hammond~\cite{ganham24} noted that the more general formula in~\eqref{eq:cov_geos} would imply that the system undergoes a transition from `stability' to `chaos' at $t\asymp\frac{1}{|v|}\Var(T(0,v))$, in that~\eqref{eq:cov_geos} together with~\eqref{eq:linear} and the monotonicity of the integrand imply that the following holds, as $|v|\to\infty$.
\begin{enumerate}[\quad (i)]
\item For $t\ll\frac{1}{|v|}\Var(T(0,v))$ we have $\Corr\big(T_0(0,v),T_t(0,v)\big)\ge1-o(1)$.
\item For $t\gg\frac{1}{|v|}\Var(T(0,v))$ we have $\E\big[\big|\pi_0(0,v)\cap\pi_t(0,v)\big|\big]\le o(|v|)$.
\end{enumerate}
(We show how to derive a more precise version of (i)-(ii) from~\eqref{eq:cov_geos} in Proposition~\ref{prop:transition} below.) Consequently, in view of the variance bound~\eqref{eq:sublin}, establishing~\eqref{eq:cov_geos} will suffice in order to establish chaos in first-passage percolation for fixed $t>0$, and a transition from stability to chaos at $t\asymp\frac{1}{|v|}\Var(T(0,v))$.
Our first theorem achieves precisely this in the integer-valued setting.

\begin{theorem}\label{th:integer}
If $F$ is supported on $\{0,1,2,\ldots\}$ and satisfies~\eqref{eq:pc_moment}, then~\eqref{eq:cov_geos} holds for $t\in[0,1]$, and the system transitions from stability to chaos at $t\asymp\frac{1}{|v|}\Var(T(0,v))$ in the sense that (i)-(ii) above hold. Moreover, if also~\eqref{eq:var_cond} holds, then we have for $t\gg1/\log|v|$ that as $|v|\to\infty$
$$
\E\big[\big|\pi_0(0,v)\cap \pi_t(0,v)\big|\big]=o(|v|).
$$
\end{theorem}

Our second result covers all continuous distributions, and more generally all weight distributions that do not have any mass at the infimum of its support, i.e.\ satisfying $F(r)=0$, where
$$
r:=\inf\{x\ge0:F(x)>0\}.
$$
In this case we have not been able to establish a relation between fluctuations and geodesics that is quite as precise as~\eqref{eq:fluct_geos} and~\eqref{eq:cov_geos}. It is nevertheless sufficient to deduce that the transition from stability to chaos occurs at $t\asymp\frac{1}{|v|}\Var(T)$ also here.

\begin{theorem}\label{th:cont} If $F$ has finite second moment and $F(r)=0$, then, for fixed $\varepsilon>0$, we have\footnote{Here the relation $a_v\asymp b_v\pm\vep|v|$ is said to hold for fixed $\vep>0$ if there exist constants $0<C_1,C_2<\infty$, possibly depending on $\vep$, such that $C_1(b_v-\vep|v|)\leq a_v\leq C_2(b_v+\vep|v|)$ for all $v\in\Z^d$.}
\begin{equation}\label{eq:th_cont}
\Var\big(T(0,v)\big)\asymp \int_0^1\E\big[\big|\pi_0(0,v)\cap \pi_t(0,v)\big|\big]\,dt\pm \vep|v|.
\end{equation}
Moreover, the system undergoes a transition from stability to chaos at $t\asymp\frac{1}{|v|}\Var(T)$ in the sense that (i)-(ii) above hold.
If also~\eqref{eq:var_cond} holds, then we have, in particular, for $t\gg1/\log|v|$ that as $|v|\to\infty$
$$\E\big[\big|\pi_0(0,v)\cap \pi_t(0,v)\big|\big]=o(|v|).$$
\end{theorem}

We believe that the relations~\eqref{eq:fluct_geos} and~\eqref{eq:cov_geos} hold without error term also in the setting of Theorem~\ref{th:cont}, but a more detailed analysis may be required to verify this.

Our third result shows that the chaotic behaviour of the first-passage metric implies a `multiple valleys' property, i.e.\ that between any pair of distant points there are many almost disjoint paths that are almost optimal. The almost optimal paths correspond to multiple local valleys in the `energy landscape' that reach close to the depth of the global minimum. For simplicity, we state this result only in the case when $F$ is continuous, so that there is an almost surely unique geodesic between any pair of points.

We write $T(\Gamma):=\sum_{e\in\Gamma}\omega(e)$ for the passage time of a path $\Gamma$.

\begin{theorem}\label{th:mv}
Assume that $F$ is continuous and satisfies~\eqref{eq:var_cond}. For every $v\in\Z^d$, with probability tending to one as $|v|\to\infty$, there exists a set of paths $\{\Gamma_i\}$ from $0$ to $v$ which satisfies $|\{\Gamma_i\}|\to\infty$, $|\Gamma_i\cap\Gamma_j|=o(|v|)$, and $T(\Gamma_i)-T(0,v)=o(|v|)$.
\end{theorem}

This work was inspired by previous work of Chatterjee~\cite{chatterjee14a} on disordered Gaussian systems. The main contribution of this paper is two-fold. First we establish a dynamical covariance formula in a general context that relates the covariance of a function to the influence of its respective variables; see Propositions~\ref{prop:cov_finite} and~\ref{prop:cov_infinite}. This part of our work is inspired by recent work of Tassion and Vanneuville~\cite{tasvan23}, who derived a related formula for Boolean functions. Second we specialise to first-passage percolation, and aim to make precise the connection between influences and geodesics that has been alluded to in work of Benjamini, Kalai and Schramm~\cite{benkalsch03}, and which would give~\eqref{eq:fluct_geos} and~\eqref{eq:cov_geos}. We are not able to achieve this in full generality, but obtain partial results in this direction. In a companion paper~\cite{ahldeisfr24} we show that this latter part poses less of a problem in the context of last-passage percolation, and establish a version of~\eqref{eq:cov_geos} for that model.

\subsection{Method of proof}

Below we use the notation $T:=T(0,v)$ and $T_t:=T_t(0,v)$ for brevity. The proofs of Theorems~\ref{th:integer} and~\ref{th:cont} rely on writing the covariance of $T_0$ and $T_t$ in terms of a certain correlation operator. Specifically, we define
$$
Q_t(T):=\E\big[T_0(0,v)T_t(0,v)\big]
$$
and note that
$$
\Cov(T_0,T_t)=Q_t(T)-Q_1(T)=-\int_t^1Q'_s(T)\,ds,
$$
assuming the derivative is well-defined. The derivative of $Q_t(T)$ can be expressed in terms of the contribution from each of the individual variables, which is typically referred to as a measure of the `influence' of the variables. More precisely, we let $\sigma_e^x:[0,\infty)^\Ec\to[0,\infty)^\Ec$ denote the operator that replaces the value of coordinate $e$ by $x$, and set
$$
D_e^xT_t:=T\circ\sigma_e^x(\omega_t)-\int T\circ\sigma_e^y(\omega_t)\,dF(y),
$$
where the integral refers to the average over the integrand with respect to coordinate $e$. We define the {\bf co-influence} of an edge $e$ with respect to $T$ at time $t$ as
$$
\Inf_e(T_0,T_t):=\int\E[D_e^xT_0D_e^xT_t]\,dF(x).
$$
A straightforward coupling argument shows that the co-influences are decreasing in $t$ (see Lemma \ref{lma:inf_inf}), and hence upper bounded by the following $L^2$-notion of influences
\begin{equation}\label{eq:L2inf}
\Inf_e(T):=\Inf_e(T_0,T_0)=\E\big[(T-\E_e[T])^2\big],
\end{equation}
where $\E_e$ refers to expectation taken over the weight at $e$. The derivative of $Q_t(T)$ turns out to equal the edge-sum of the co-influences, resulting in the covariance formula
\begin{equation}\label{eq:varform}
\Cov(T_0,T_t)=\int_t^1\sum_{e\in\Ec}\Inf_e(T_0,T_s)\,ds,
\end{equation}
which for $t=0$ reduces to a formula for the variance. As such, we obtain together with~\eqref{eq:L2inf} an Efron-Stein-like upper bound on the variance.

Formulas of a flavour similar to~\eqref{eq:varform} have appeared in the literature before. A formula in discrete time was derived by Chatterjee in~\cite{chatterjee05}, and used for a similar purpose in~\cite{borlugzhi20}. A related formula in the Gaussian setting, in which the variables update continuously according to independent Ornstein-Uhlenbeck processes, was central in~\cite{chatterjee14a}, and in the context of Boolean functions, a formula of this type was derived in~\cite{tasvan23}.

It has been suggested that the influence of an edge (with respect to $T=T(0,v)$) is roughly proportional to the probability that the edge belongs to the geodesic between $0$ and $v$. In the dynamical setting this corresponds to a statement of the form
\begin{equation}\label{eq:inf_asymp}
\Inf_e(T_0,T_t)\asymp\PP(e\in\pi_0\cap\pi_t),
\end{equation}
where $\pi_t$ is short for $\pi_t(0,v)$. Apart from the derivation of~\eqref{eq:varform}, the main goal of this paper will be to establish a version of this statement, which together with~\eqref{eq:varform} will give a version of~\eqref{eq:fluct_geos} and~\eqref{eq:cov_geos}.
Work of Ahlberg and Hoffman~\cite{ahlhof}, and Dembin, Elboim and Peled~\cite{demelbpel}, on the so-called `midpoint problem', shows that when $F$ is continuous, the influence of edges far from the endpoints tends to zero as $|v|\to\infty$. Our results thus justify the claim that the influence of an edge vanishes as $|v|\to\infty$.

\subsection{Related results}

There has been an increasing interest in the phenomenon of chaos in recent years. The first rigorous evidence of chaos was obtained by Chatterjee~\cite{chatterjee1,chatterjee2,chatterjee14a} for Gaussian disordered systems, including directed Gaussian polymers, the Sherrington-Kirkpatrick model of spin glasses, and maxima of Gaussian fields. In the mathematical study of spin glasses, Chatterjee's work has been followed up in a range of papers, both within and outside of the Gaussian setting; see e.g.~\cite{arghan20,aufche16,chatterjee3,chen13,eldan20}. Chaos has also been observed in the context of eigenvalues of random matrices. Sensitivity of the eigenvector associated with the top eigenvalue of a Wigner matrix established by Bordenave, Lugosi and Zhivotovskiy~\cite{borlugzhi20}, extending earlier work of Chatterjee~\cite{chatterjee14a} outside of the Gaussian setting. Moreover, in~\cite{borlugzhi20} the authors determine the precise order at which the transition from stability to chaos occurs, and in a more precise sense than that we consider here, which for an $n\times n$ Wigner matrix is at $t=\Theta(n^{-1/3})$. 

The literature on random matrices is extensive, and eigenvalues/-vectors thereof have a favourable structure. A significantly greater effort was required by Ganguly and Hammond~\cite{ganham23,ganham24} to establish that the transition from stability to chaos in Brownian last-passage percolation occurs at $t=\Theta(n^{-1/3})$. Their work comprises the first instance of chaos in stochastic models for spatial growth. Their work is again in a Gaussian context, but we emphasise that it is a stronger transition from stability to chaos that is considered in their work, in that both the stability and chaotic regimes concern the expected overlap of geodesics. In a companion paper~\cite{ahldeisfr24} we also extend the weaker transition, which is under discussion in this paper, to other non-Gaussian models of last-passage percolation, in order to emphasise that the transition from stability to chaos should more generally occur at $t=\Theta(\frac1n\Var(T))$, which in the so-called `exactly solvable' setting is equivalent to $t=\Theta(n^{-1/3})$.
In parallel to the preparation of this manuscript, Dembin and Garban~\cite{demgar} have given further evidence for the existence of chaos in the context of max-flow problems associated with first-passage percolation and disordered Ising ferromagnets.

Let us finally mention that a chaotic behaviour has also been observed for the set of pivotal edges in planar percolation models. In the conformally invariant setting, i.e.\ site percolation on the triangular lattice, the pivotal set corresponding to the crossing of an $n\times n$-box exhibits chaos, with the transition occurring at $t=\Theta(n^{-3/4})$. This can be derived as a consequence of the precise analysis of dynamical four-arm events due to Tassion and Vanneuville~\cite{tasvan23}, although parts of this conclusion can also be derived from earlier work~\cite{garpetsch18}; see also~\cite{galpet}.

\subsection{Future work and outline}

In this paper we establish a transition from stability to chaos of the first-passage metric in a weak sense. In~\cite{borlugzhi20,ganham24} a stronger transition is established in that both the sable and chaotic regimes consider the same object (eigenvector and geodesics, respectively). We conjecture that the same stronger transition holds in first-passage percolation.

\begin{conjecture}\label{conj:1}
For $t\ll\frac{1}{|v|}\Var(T(0,v))$ we have $\E\big[\big|\pi_0(0,v)\cap\pi_t(0,v)\big|\big]=\Theta(|v|)$.
\end{conjecture}

Similarly, we expect that in the chaotic regime the travel times decorrelate. However, in this case we are not aware of results of this kind having been establised for related models.

\begin{conjecture}\label{conj:2}
For $t\gg\frac{1}{|v|}\Var(T(0,v))$ we have $\Corr\big(T_0(0,v),T_t(0,v)\big)=o(1)$.
\end{conjecture}

The concept of `noise sensitivity' was introduced in the context of Boolean functions by Benjamini, Kalai and Schramm~\cite{benkalsch99}, and refers to the asymptotic independence of the output of a Boolean function with respect to two highly correlated inputs. In the context of first-passage percolation, the analogous notion of noise sensitivity would refer to decorrelation (for all fixed $t>0$) in the sense of Conjecture~\ref{conj:2}, i.e.\
\begin{equation}\label{eq:ns}
\Corr(T_0,T_t)\to0\quad\text{as }|v|\to\infty,
\end{equation}
for all $t>0$. 
Even for fixed $t>0$,~\eqref{eq:ns} is not known in first-passage percolation, nor is it known to hold in other models of spatial growth. From Proposition~\ref{prop:cov_infinite} below it follows (by restricting the integral to $[t,2t]$) that
$$
\sum_{e\in\Ec}\Inf_e(T_0,T_{2t})\le\frac1t\Cov(T_0,T_t),
$$
which shows that~\eqref{eq:ns} would imply that $\sum_{e\in\Ec}\Inf_e(T_0,T_{2t})=o(\Var(T))$. This suggests, in combination with~\eqref{eq:inf_asymp}, that noise sensitivity of $T$ is a stronger expression of a small perturbation than the chaotic behaviour derived here.

In this paper we establish an instance of~\eqref{eq:fluct_geos}, and use this relation to find evidence for chaos in the first-passage metric. It would be interesting to improve upon the bounds on the overlap between geodesics we obtain here, and conversely use~\eqref{eq:fluct_geos} for the central task to improve upon existing variance bounds.

The rest of the paper is organised as follows: We first establish a general covariance formula extending~\eqref{eq:varform} in Sections~\ref{sec:corr} and~\ref{sec:var}. We examine the notion of co-influence in the context of first-passage percolation in Section~\ref{sec:influences}, and prepare for the proofs of Theorems~\ref{th:integer} and~\ref{th:cont}. The proof in the integer-valued setting is simpler, due to the fact that \emph{if} changing the weight of an edge has an effect on the distance between two points, then the magnitude of the effect has to be at least one. Theorem~\ref{th:integer} is proved in Section~\ref{sec:integer} and Theorem~\ref{th:cont} is proved in Section~\ref{sec:cont}. Finally we deduce Theorem~\ref{th:mv} as a corollary in Section~\ref{sec:mv}.

\section{A dynamical covariance formula}\label{sec:corr}

In this section we derive a finitary version of the dynamical variance formula, stated in~\eqref{eq:varform}, that will be central for the remaining analysis. The corresponding formula in infinite volume will be derived in the next section. Since we shall need to consider various functions of the edge weights apart from the first-passage time, and since formulas of this kind have a general interest, we shall develop the theory in a general setting. We begin with some notation.

Let $m\ge1$ be an integer and $F$ (the distribution function of) some probability measure on $\RR$. Let $\omega=(\omega(i))_{i\in[m]}$ and $\omega'=(\omega'(i))_{i\in[m]}$ be two independent random vectors with i.i.d\ coordinates with common distribution $F$. Let $U(1),U(2),\ldots,U(m)$ be independent variables uniformly distributed on the interval $[0,1]$. For $t\in[0,1]$ and $i=1,2,\ldots,m$, set
\begin{equation}\label{eq:omega_t}
\omega_t(i):=\left\{
\begin{aligned}
\omega(i) &&& \text{if }U(i)>t,\\
\omega'(i) &&& \text{if }U(i)\le t.
\end{aligned}
\right.
\end{equation}
We shall write $\PP$ for the associated probability measure, and $\E$ for its expectation.

We will derive a dynamical covariance formula for functions in $L^2(F^m)$, i.e.\ the collection of real-valued functions $f:\RR^m\to\RR$ such that
$$
\E[f(\omega)^2]=\int f^2\,dF^m<\infty.
$$
Let $\sigma_i^x:\RR^m\to\RR^m$ denote the operator that replaces the value of the $i$th coordinate with $x$. Hence, $f\circ\sigma_i^x$ evaluates $f$ with the $i$th coordinate fixed to be $x$. We introduce the following notation for the difference between $f$ and its average over the $i$th coordinate: For $x\in\RR$ let
\begin{align}
\label{eq:Df}
D_i^xf(\omega)&:=f\circ\sigma_i^x(\omega)-\int f\circ\sigma_i^y(\omega)\,dF(y),\\
\label{eq:Df2}
D_if(\omega)&:=f(\omega)-\int f\circ\sigma_i^y(\omega)\,dF(y).
\end{align}
We define the {\bf co-influence} of coordinate $i$ for the function $f$ at time $t$ as
\begin{equation}\label{eq:influences}
\Inf_i\big(f(\omega_0),f(\omega_t)\big):=\int \E\big[D_i^xf(\omega_0)D_i^xf(\omega_t)\big]\,dF(x).
\end{equation}

We remark that the differences in~\eqref{eq:Df}-\eqref{eq:Df2} are not necessarily well-defined for every $\omega\in\RR^m$, since the average over the $i$th coordinate is only guaranteed to exist almost surely. However, we note that
\begin{equation}\label{eq:DfL2}
\E[(D_if(\omega))^2]\le\E[f(\omega)^2]<\infty,
\end{equation}
so that $D_if$ is again in $L^2(F^m)$. It follows (see also Lemma~\ref{lma:influence} below) that $D_i^xf$ is in $L^2(F^m)$ for $F$-almost every $x\in\RR$, and so the co-influences are well-defined for every $t\in[0,1]$. We note further that
$$
\Inf_i\big(f(\omega_0),f(\omega_0)\big)=\E[(D_if)^2],
$$
which recovers the $L^2$-notion of the influence of a coordinate as in~\eqref{eq:L2inf}, in analogy to its appearance in the study of Boolean functions~\cite{garste15,odonnell14}.

The main result of this section is a formula that relates the dynamical covariance of a function $f$ to the (dynamic) co-influences of the respective coordinates.

\begin{prop}\label{prop:cov_finite}
For every $f\in L^2(F^m)$ the covariance $\Cov(f(\omega_0),f(\omega_t))$ is non-negative and non-increasing as a function of $t$ and, for $t\in[0,1]$, it satisfies 
$$
\Cov\big(f(\omega_0),f(\omega_t)\big)=\int_t^1\sum_{i=1}^m\Inf_i\big(f(\omega_0),f(\omega_s)\big)\,ds.
$$
\end{prop}

We shall prove the proposition in two intermediary steps, using coupling techniques similar to those in~\cite{tasvan23}. To guide the proof, we introduce the notation
$$
Q_t(f):=\E[f(\omega_0)f(\omega_t)].
$$
In this notation, we have
\begin{equation}\label{eq:cov_int_diff}
\Cov\big(f(\omega_0),f(\omega_t)\big)=Q_t(f)-Q_1(f)=\int_t^1-Q'_s\,ds,
\end{equation}
under the assumption that $Q_t(f)$ is continuously differentiable (which we verify below).

As a first step we prove the following lemma, using a standard coupling argument.

\begin{lemma}\label{lma:monotonicity}
For every $f\in L^2(F^m)$ and $0\le s\le t\le1$ we have
$$
Q_s(f)\ge Q_t(f)\ge0.
$$
\end{lemma}

\begin{proof}
Let $0\le s\le t\le1$ be fixed, and let $V$ be a random vector with i.i.d.\ coordinates distributed as $F$. We shall construct new random vectors $X$, $Y$ and $Z$ from $V$ through a resampling procedure. Since the resampling procedure acts independently on each coordinate, the resulting vectors are again i.i.d.\ with marginal distribution $F$.

Set
$$
p=1-\sqrt{1-s},\; q=1-\sqrt{1-t}\text{ and }r=\frac{q-p}{1-p}.
$$
Since by assumption $0\le s\le t\le 1$, it follows that $0\le p\le q\le 1$, and hence that $0\le r\le1$.

Let $W,W',W'',W'''$ be independent copies of $V$. First, let $V'$ be obtained from $V$ by independently replacing each coordinate of $V$ with the corresponding coordinate from $W$ with probability $r$. Next, we obtain $X$ from $V'$ by independently replacing each coordinate of $V'$ with the corresponding coordinate from $W'$ with probability $p$, and similarly obtain $Y$ from $V'$ by independently replacing each coordinate of $V'$ with the corresponding coordinate from $W''$ with probability $p$. Finally, we obtain $Z$ from $V$ by independently replacing each coordinate with the corresponding coordinate of $W'''$ with probability $q$.

A straightforward calculation shows that the joint distribution of the pair $(X,Y)$ is equal to that of $(\omega_0,\omega_s)$, and that the joint distribution of $(X,Z)$ is equal to that of $(\omega_0,\omega_t)$. Moreover, $X$ and $Y$ are conditionally independent given $V'$, and $X$ and $Z$ are conditionally independent given $V$. Therefore, it follows that
\begin{align*}
Q_s(f)&=\E[f(X)f(Y)]=\E\big[\E[f(X)f(Y)|V']\big]=\E\big[\E[f(X)|V']^2\big]\ge0,\\
Q_t(f)&=\E[f(X)f(Z)]=\E\big[\E[f(X)f(Z)|V]\big]=\E\big[\E[f(X)|V]^2\big]\ge0.
\end{align*}

Let $\mathcal{F}$ denote the $\sigma$-algebra of information generated by $V$, $W$ and the randomness used to obtain $V'$ from $V$. Both $V$ and $V'$ are measurable with respect to $\mathcal{F}$, but $\mathcal{F}$ contains no more information regarding $X$. Hence, it follows from Jensen's inequality that
$$
Q_t(f)=\E\big[\E[f(X)|V]^2\big]=\E\big[\E[\,\E[f(X)|\mathcal{F}]\,|V]^2\big]\le\E\big[\E[f(X)|\mathcal{F}]^2\big]=Q_s(f),
$$
as required.
\end{proof}

From the previous lemma we obtain an analogous statement for the co-influences.

\begin{lemma}\label{lma:influence}
For every $f\in L^2(F^m)$ and $i=1,2,\ldots,m$ we have $D_i^xf\in L^2(F^m)$ for $F$-almost every $x\in\RR$. Moreover, for $0\le s\le t\le1$ the co-influences are finite and
$$
\Inf_i(f(\omega_0),f(\omega_s))\ge\Inf_i(f(\omega_0),f(\omega_t))\ge0.
$$
\end{lemma}

\begin{proof}
Using Cauchy-Schwartz' inequality and~\eqref{eq:DfL2} we obtain that
$$
\Inf_i(f(\omega_0),f(\omega_t))\le\int\E\big[(D_i^xf(\omega))^2\big]\,dF(x)=\E\big[(D_if(\omega))^2\big]<\infty.
$$
In particular, the influences are well-defined and finite, and $D_i^xf$ is in $L^2(F^m)$ for $F$-almost every $x\in\RR$. Moreover,
$$
\Inf_i(f(\omega_0),f(\omega_t))=\int\E\big[D_i^xf(\omega_0)D_i^xf(\omega_t)\big]\,dF(x)=\int Q_t(D_i^xf)\,dF(x).
$$
By Lemma~\ref{lma:monotonicity} we have for $F$-almost every $x\in\RR$ that $Q_t(D_i^xf)$ is non-negative and non-decreasing in $t$. These properties are preserved under integration (over $x$), so the proof is complete.
\end{proof}

We next move to relate the derivative of $Q_t(f)$ with the co-influences at time $t$. For proof-technical reasons we shall extend the construction so that a different time variable is associated with each coordinate. Let $X=(X_1,\ldots, X_m)$ and $Z=(Z_1,\ldots ,Z_m)$ be two independent $\RR^m$-valued random vectors with i.i.d.\ components distributed as $F$, and let $U_1,\ldots, U_m$ be independent uniform random variables on $[0,1]$. Given $t_1, \dots, t_m \in [0,1]$, let $Y(t_1, \dots, t_m)$ be the $m$-dimensional random vector, where the $i$th coordinate is given by
\begin{equation}\label{eq:XYZ}
Y_i(t_i) = \begin{cases}
 X_i & \text{ if } U_i > t_i, \\
 Z_i & \text{ if } U_i \le t_i.
\end{cases}
\end{equation}
The vector $Y(t_1, \dots, t_m)$ is hence obtained from $X$ by resampling the $i$th coordinate with probability $t_i$, then replacing it with the corresponding entry from $Z$. When all $t_i$ are equal, we write $Y_t$ for $Y(t,\ldots, t)$. In particular, $(Y_0,Y_t)$ and $(\omega_0,\omega_t)$ are equal in distribution and, for $f$ in $L^2(F^m)$, we have
$$
Q_t(f)=\E[f(Y_0)f(Y_t)].
$$

\begin{lemma}\label{lma:derivative}
For every $f\in L^2(F^m)$ and $i=1,2,\ldots,m$ we have
$$
- \frac{\partial}{\partial t_i} \mathbb{E} \big[f(X)f(Y(t_1, \dots, t_m))\big] = \int\mathbb{E} \big[D_i^xf(X) D_i^x (Y(t_1, \dots, t_m))\big] \, dF(x),
$$
and the right-hand side is continuous in each of its coordinates.
\end{lemma}

\begin{proof}
The following argument is identical for all $i\in[m]$ so we consider only $i=1$. We then want to compute the derivative of the function
$$
\phi(t_1):= \mathbb{E}\big[f(X)f(Y(t_1, t_2, \dots, t_m))\big].
$$
To this end, fix $t_1\in(0,1)$ and let $\delta>0$ be such that $t_1+\delta\in[0,1]$. Write $Y=Y(t_1, \dots, t_m)$ and $Y' = Y(t_1 + \delta, t_2, \dots, t_m)$, and observe that
\begin{equation*}
\phi(t_1+\delta) - \phi(t_1) = \mathbb{E}\big[f(X)\big(f(Y')-f(Y)\big)\big].
\end{equation*}
Note that the random variables $Y$ and $Y'$ can differ only in their first coordinate. They do so when $U_1 \in (t_1, t_1+\delta]$, in which case $Y_1 = X_1$ and $Y_1' = Z_1$. Write $Y^X=(X_1,Y_2(t_2),\ldots,Y_m(t_m))$ and $Y^Z=(Z_1,Y_2(t_2),\ldots,Y_m(t_m))$. Since $U_1$ takes values in $(t_1,t_1+\delta]$ with probability $\delta$, we have
\begin{equation}\label{eq:delta_prob}
\phi(t_1+\delta) - \phi(t_1) = \delta\E\big[f(X)\big(f(Y^Z)-f(Y^X)\big)\big].
\end{equation}
Since $Y^Z$ and $Y^X$ differ only in the first coordinate, and since $f(X)-D_1f(X)$ is determined by $(X_2,\ldots,X_m)$, it follows (by averaging over $X_1$ and $Z_1$ first) that
$$
\phi(t_1+\delta)-\phi(t_1)=\delta\E\big[D_1f(X)\big(f(Y^Z)-f(Y^X)\big)\big].
$$
Moreover, $f(Y^Z)-f(Y^X)=D_1f(Y^Z)-D_1f(Y^X)$, and since $\E[D_1f(X)|(X_2,\ldots,X_m)]$ is zero, we conclude that
$$
\phi(t_1+\delta)-\phi(t_1)=\delta\E\big[D_1f(X)\big(D_1f(Y^Z)-D_1f(Y^X)\big)\big]=-\delta\E\big[D_1f(X)D_1f(Y^X)\big].
$$

An analogous computation yields the same conclusion for $\delta<0$, but with the opposite sign. (For $\delta<0$, the variable $U_1$ takes values in $(t_1+\delta,t_1]$ with probability $-\delta$, which gives an additional `$-$' in the right-hand side of~\eqref{eq:delta_prob}.) Hence
$$
- \phi'(t_1) = \mathbb{E} \big[D_1f(X) D_1f(Y^X)\big] = \int \mathbb{E} [D_1^x f(X) D_1^x f(Y(t_1, \dots, t_n))] \, dF(x),
$$
as required. Finally, we note that $\mathbb{E} \big[D_1f(X) D_1f(Y^X)\big]$ is constant in $t_1$, and that an argument analogous to that leading to~\eqref{eq:delta_prob} shows that it is continuous in the remaining coordinates $t_2,\ldots,t_m$.
\end{proof}

We are finally set to prove the main result of this section.

\begin{proof}[Proof of Proposition~\ref{prop:cov_finite}]
First, recall that $\Cov(f(\omega_0),f(\omega_t))=Q_t(f)-Q_1(f)$, so that non-negativity and (weak) monotonicity of the covariance follows from the monotonicity of $Q_t(f)$, which was proved in Lemma~\ref{lma:monotonicity}.

Second, Lemma~\ref{lma:derivative} and the chain rule show that $Q_t(f)$ is continuously differentiable on $(0,1)$ with
\begin{equation}\label{eq:derivative}
-Q'_t(f)=\sum_{i=1}^m\int\mathbb{E} \big[D_i^xf(Y_0) D_i^x (Y_t)\big] \, dF(x)= \sum_{i=1}^m\textup{Inf}_i\big(f(\omega_0),f(\omega_t)\big).
\end{equation}
Hence, by the fundamental theorem of calculus,~\eqref{eq:cov_int_diff} is justified, which together with~\eqref{eq:derivative} completes the proof.
\end{proof}

\section{A dynamical covariance formula in infinite volume}\label{sec:var}

In the previous section we derived a dynamical covariance formula for real-valued functions on $\RR^m$. In this section we extend the covariance formula to a formula for functions of infinitely many variables, and hence establish~\eqref{eq:varform}. We shall again work with general functions, as we shall need to apply parts of the theory to functions other than the distance function. Since variables associated with different edges of the lattice are independent, and since there is a bijection between $\Z^d$ and the natural numbers, and between the set of nearest-neighbour edges of $\Z^d$ and the natural numbers, it will suffice to establish a formula for functions $f:\RR^\N\to\RR$ that satisfy
$$
\E[f(\omega)^2]=\int f^2\,dF^\N<\infty,
$$
where again $F$ denotes some probability distribution on $\RR$.
We shall denote this class of functions by $L^2(F^\N)$, and use the convention that $\N=\{1,2,\ldots\}$.

Definitions of central concepts introduced in the previous section extend straightforwardly to the infinite setting. Let $\omega$ and $\omega'$ be independent $\RR^\N$-valued random vectors with i.i.d.\ coordinates distributed as $F$. Let $U(1),U(2),\ldots$ be independent uniform random variables on $[0,1]$, and let $\omega_t$ be the $\RR^\N$-valued random vector defined as in~\eqref{eq:omega_t}. Note that the definitions of differences and co-influences in~\eqref{eq:Df}-\eqref{eq:influences} now extend straightforwardly to function in $L^2(F^\N)$.

The main result of this section is the following dynamic covariance formula.

\begin{prop}\label{prop:cov_infinite}
For every $f\in L^2(F^\N)$ the covariance $\Cov(f(\omega_0),f(\omega_t))$ is non-negative and non-increasing as a function of $t$ and, for $t\in[0,1]$, it satisfies 
$$
\Cov\big(f(\omega_0),f(\omega_t)\big)=\int_t^1\sum_{i=1}^\infty\Inf_i\big(f(\omega_0),f(\omega_s)\big)\,ds.
$$
\end{prop}

Again, since edge weights are i.i.d.\ and there is a bijection between the set of edges of the nearest-neighbour lattice and the natural numbers,
we note that the variance formula in~\eqref{eq:varform} is an immediate consequence of Proposition~\ref{prop:cov_infinite} by taking $t=0$.

Our first lemma is an infinite volume version of Lemma~\ref{lma:monotonicity}, and is proved analogously. Since we shall need a version of the lemma where different coordinates are adjusted separately, we reuse the notation from the previous section. Let $X$ and $Z$ be two independent $\RR^\N$-valued random vectors with i.i.d.\ components distributed as $F$, and let $U_1,U_2,\ldots$ be independent uniform random variables on $[0,1]$. For $t_1,t_2,\ldots\in[0,1]$, let $Y(t_1,t_2,\ldots)$ be the $\RR^\N$-valued random vector whose coordinates are defined as in~\eqref{eq:XYZ}, and write $Y_t$ for $Y(t,t,\ldots)$ for compactness. As before, $(Y_0,Y_t)$ and $(\omega_0,\omega_t)$ have the same distribution, and we let
$$
Q_t(f):=\E[f(\omega_0)f(\omega_t)]=\E[f(Y_0)f(Y_t)].
$$

\begin{lemma}\label{lma:mono_inf}
For every $f\in L^2(F^\N)$ and $0\le s_i\le t_i\le1$ for $i=1,2,\ldots$, we have
$$
\E\big[f(X)f(Y(s_1,s_2,\ldots))\big]\ge\E\big[f(X)f(Y(t_1,t_2,\ldots))\big]\ge0.
$$
In particular, for $0\le s\le t\le 1$, we have $Q_s(f)\ge Q_t(f)\ge0$.
\end{lemma}

\begin{proof}
This follows from a straightforward adaptation of the proof of Lemma~\ref{lma:monotonicity}, where the probabilities $p$, $q$ and $r$ are now coordinate dependent.
\end{proof}

We similarly obtain the following extension of Lemma~\ref{lma:influence}.

\begin{lemma}\label{lma:inf_inf}
For every $f\in L^2(F^\N)$ and $i=1,2,\ldots$, we have $D_i^xf\in L^2(F^\N)$ for $F$-almost every $x\in\RR$. Moreover, for $0\le s\le t\le 1$, the co-influences are finite and satisfy
$$
\Inf_i(f(\omega_0),f(\omega_s))\ge\Inf_i(f(\omega_0),f(\omega_t))\ge0.
$$
\end{lemma}

\begin{proof}
The proof is identical to the proof of Lemma~\ref{lma:influence}.
\end{proof}

We are now in a position to derive the main result of this section.

\begin{proof}[Proof of Proposition~\ref{prop:cov_infinite}]
That the covariance is non-negative and non-increasing is immediate from Lemma~\ref{lma:mono_inf} and the fact that $\Cov(f(\omega_0),f(\omega_t))=Q_t(f)-Q_1(f)$. We shall derive the remaining formula by approximating $f$ by the function obtained by averaging over all but the first $m$ coordinates, and apply Proposition~\ref{prop:cov_finite}.

Let $\mathcal{F}_m$ denote the $\sigma$-algebra generated by $\{\omega(i),\omega'(i),U(i):i=1,2,\ldots,m\}$, and set
$$
h_m(\omega_t):=\E[f(\omega_t)|\mathcal{F}_m].
$$
Since $f\in L^2(F^\N)$, the sequence $(h_m)_{m\ge1}$ forms an $L^2$-bounded martingale, and by Levy's upwards theorem, as $m\to\infty$, we have that
$$
h_m(\omega_t)\to f(\omega_t)\quad\text{almost surely and in }L^2.
$$
It follows, in particular, that
\begin{equation}\label{eq:cov_lim}
\Cov\big(h_m(\omega_0),h_m(\omega_t)\big)\to\Cov\big(f(\omega_0),f(\omega_t)\big),\quad\text{as }m\to\infty.
\end{equation}
Moreover, applying Proposition~\ref{prop:cov_finite}, we obtain that
\begin{equation}\label{eq:cov_fin}
\Cov\big(h_m(\omega_0),h_m(\omega_t)\big)=\int_t^1\sum_{i=1}^m\Inf_i\big(h_m(\omega_0),h_m(\omega_s)\big)\,ds.
\end{equation}

We next want to argue that, for every $i\in\N$, $m\ge i$ and $s\in[0,1]$, we have
\begin{equation}\label{eq:inf_mon}
\Inf_i\big(h_m(\omega_0),h_m(\omega_s)\big)\to\Inf_i\big(f(\omega_0),f(\omega_s)\big)\quad\text{as }m\to\infty,
\end{equation}
and that the convergence is monotone in $m$. (For $m<i$ we have $\Inf_i(h_m(\omega_0),h_m(\omega_s))=0$, but this is not a crucial observation.)
Once this has been established, since the co-influences are non-negative according to Lemma~\ref{lma:influence}, it follows from the monotone convergence theorem, together with~\eqref{eq:cov_lim} and~\eqref{eq:cov_fin}, that
$$
\Cov\big(f(\omega_0),f(\omega_t)\big)=\lim_{m\to\infty}\Cov\big(h_m(\omega_0),h_m(\omega_t)\big)=\int_t^1\sum_{i=1}^\infty\Inf_i\big(f(\omega_0),f(\omega_s)\big)\,ds,
$$
as required.

It remains to establish the monotone convergence in~\eqref{eq:inf_mon}.
By Lemma~\ref{lma:inf_inf}, we have $D_i^xf\in L^2(F^\N)$ for $F$-almost every $x\in\RR$. For these values of $x$, Fubini's theorem shows that
\begin{equation}\label{eq:Dh_fubini}
D_i^xh_m(\omega_s)=\E[D_i^xf(\omega_s)|\mathcal{F}_m]\quad\text{almost surely},
\end{equation}
so that, by Levy's upwards theorem, as $m\to\infty$
$$
D_i^xh_m(\omega_s)\to D_i^xf(\omega_s)\quad\text{almost surely and in }L^2.
$$
We conclude in particular that, for $F$-almost every $x$, we have
\begin{equation}\label{eq:EDh_lim}
\E[D_i^xh_m(\omega_0)D_i^xh_m(\omega_s)]\to\E[D_i^xf(\omega_0)D_i^xf(\omega_s)]\quad\text{as }m\to\infty.
\end{equation}

To see that the convergence in~\eqref{eq:EDh_lim} is monotone, let $Y^{(m)}$ denote the vector $Y(t_1,t_2,\ldots)$ where $t_i=s$ for $i\le m$ and $t_i=1$ for $i>m$. By~\eqref{eq:Dh_fubini} and the independence between $X$ and $Y^{(m)}$ in coordinates $i>m$, we may rewrite the left-hand side of~\eqref{eq:EDh_lim} as
$$
\E[D_i^xh_m(\omega_0)D_i^xh_m(\omega_s)]=\E\big[\E[D_i^xf(\omega_0)|\mathcal{F}_m]\E[D_i^xf(\omega_s)|\mathcal{F}_m]\big]=\E[D_i^xf(X)D_i^xf(Y^{(m)})],
$$
for those $x$ for which $D_i^xf$ is in $L^2(F^\N)$. For these values of $x$, the expression is non-negative and non-decreasing in $m$, according to Lemma~\ref{lma:mono_inf}, so the convergence in~\eqref{eq:EDh_lim} is indeed monotone for $F$-almost every $x$. Consequently, it follows from~\eqref{eq:EDh_lim} and the monotone convergence theorem that, as $m\to\infty$, we have
\begin{equation}\label{eq:inf_pre_mon}
\Inf_i\big(h_m(\omega_0),h_m(\omega_s)\big)=\int\E[D_i^xh_m(\omega_0)D_i^xh_m(\omega_s)]\,dF(x)\to\Inf_i\big(f(\omega_0),f(\omega_s)\big).
\end{equation}
Moreover, since the left-hand side of~\eqref{eq:EDh_lim} is monotone in $m$ for $F$-almost every $x$, it follows that the left-hand side in~\eqref{eq:inf_pre_mon} is again monotone in $m$. Hence, the convergence in~\eqref{eq:inf_pre_mon} is monotone, which proves~\eqref{eq:inf_mon}.
\end{proof}

\section{Influences in first-passage percolation}\label{sec:influences}

The dynamic covariance formula in Proposition~\ref{prop:cov_infinite} connects the covariance of a function to the co-influences of its variables. We now move on to analyse the co-influences in the setting of the first-passage percolation. Our aim will be to establish the connection between fluctuations and geodesics described in~\eqref{eq:fluct_geos} and~\eqref{eq:cov_geos}. In this section we make some preliminary observations, and fix notation that will be used throughout.

Henceforth, $F$ will denote (the distribution function of) some probability measure on $[0,\infty)$, used to sample the edge weights of the $\Z^d$ lattice. For ease of notation, we henceforth fix $v\in\Z^d$ and let $T:=T(0,v)=T(0,v)(\omega)$ denote the distance between $0$ and $v$ in the configuration $\omega$. It is well-known that finite second moment for $F$ is sufficient for the first-passage time $T=T(0,v)$, for every $v\in\Z^d$, to have finite second moment too. In the above notation, Proposition~\ref{prop:cov_infinite} reduces to
\begin{equation}\label{eq:VarT}
\Cov(T_0,T_t)=\int_t^1\sum_{e\in\Ec}\Inf_e(T_0,T_s)\,ds,
\end{equation}
where, as before, $T_t:=T(\omega_t)=T(0,v)(\omega_t)$ and $\omega_t$ denotes the dynamical edge weight as defined in~\eqref{passage_times}. Taking $t=0$ we recover a formula for the variance, and by expanding the square we find that
\begin{equation}\label{eq:expansion}
\E\big[(T_0-T_t)^2\big]=2\Var(T)-2\Cov(T_0,T_t)=2\int_0^t\sum_{e\in\Ec}\Inf_e(T_0,T_s)\,ds.
\end{equation}

Since any two vertices of $\Z^d$ (assuming $d\ge2$) are connected by (at least) two disjoint paths, it follows that the passage time $T(\sigma_e^x\omega)$ between $0$ and $v$ in the configuration $\sigma_e^x\omega$ is bounded as a function of $x$, for every $\omega$ and $e$. As a consequence, also $D_e^xT(\omega)$ takes a well-defined finite value for every $x$ and $\omega$ and, by definition (in~\eqref{eq:Df}), this function integrates to zero, i.e.\
\begin{equation}\label{eq:D_balance}
\int D_e^xT\,dF(x)=0
\end{equation}
for every edge $e$ and weight configuration $\omega$.

Recall that $r$ denotes the infimum of the support of $F$. We shall typically think of $x\mapsto D_e^xT$ as a function from $[r,\infty)\to\RR$. We note that $T(\omega)$ is monotone (non-decreasing) in each of its coordinates $\omega(e)$, from which it follows that $D_e^xT$ is monotone in $x$, for every $e$ and $\omega$. In fact, either $D_e^xT$ is constant equal to zero (as a function of $x$), in case there is a path not using $e$ which is optimal (for the configuration $\sigma_e^x\omega$) for all values of $x\ge r$, or $D_e^xT$ grows linearly for small values of $x$, and is constant for large values of $x$; see Figure~\ref{fig:DTt}.

Since $D_e^xT_t$ is monotone and constant for large $x$, there is a point at which increasing $x$ no longer has an effect on $D_e^xT_t$. We denote this point (depicted in Figure~\ref{fig:DTt}) by $Z_t(e)$, i.e.\
\begin{equation*}
Z_t(e):=\min\{x\ge r:D_e^xT_t=D_e^yT_t\text{ for all }y\ge x\}.
\end{equation*}
In particular, we have $r\le Z_t(e)<\infty$ for every weight configuration $\omega_t$. Since $D_e^xT_t$ grows linearly for $x\le Z_t(e)$ and is constant for $x>Z_t(e)$, it must be of the form $C-(Z_t(e)-x)_+$ for some $C=C(\omega_t)$ that does not depend on $x$, and $f_+$ denotes the positive part of the function $f$. Since $D_e^xT_t$ integrates to zero, it follows from~\eqref{eq:D_balance} that
\begin{equation}\label{eq:DZ}
D_e^xT_t=\int(Z_t(e)-y)_+\,dF(y)-(Z_t(e)-x)_+.
\end{equation}

\begin{figure}[htbp]
\begin{center}
\begin{tikzpicture}[scale=.5]
\draw (0,-4) -- (0,3);
\draw[->] (-1,0) -- (14.5,0);
\draw[draw=blue,thick] (0,-3) -- (4.5,1.5) -- (14,1.5);
\draw[dashed] (4.5,0) -- (4.5,1.5);
\draw (4.5,-.25) -- (4.5,.25);
\draw (4.5,-.25) node[anchor=north] {$Z_t(e)$};
\draw (9,1.5) node[anchor=south] {$D_e^xT_t$};
\draw (0,3) node[anchor=south] {$x=r$};
\end{tikzpicture}
\end{center}
\caption{\small Plot of the function $x \mapsto D_e^x T_t$ over the interval $[r, \infty)$.}
\label{fig:DTt}
\end{figure}
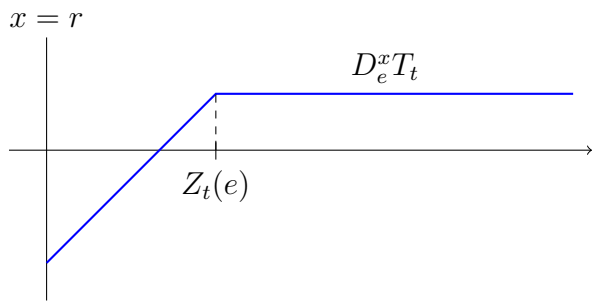

An immediate consequence of~\eqref{eq:DZ} is that
\begin{equation}\label{eq:DTmu}
D_e^xT_t\ge\int(Z_t(e)-y)\,dF(y)-Z_t(e)=-\mu,
\end{equation}
where $\mu$ is the mean of $F$. Since $D_e^xT_t$ is constant equal to zero when $Z_t(e)=r$, we obtain
\begin{equation}\label{eq:DT_bound}
|D_e^xT_t|\le(\mu+x)\mathbf{1}_{\{Z_t(e)>r\}},
\end{equation}
which can be used as a basis for upper bounding the co-influence of $e$.

Although we shall not make use of the following observation in this paper, we note that the representation in~\eqref{eq:DZ} allows us to interpret the co-influence as the covariance over the given coordinate, in the sense described next. Let
$$
\mathcal{F}_e:=\sigma\big(\{\omega(e'),\omega'(e'),U(e'):e'\neq e\}\big),
$$
i.e.\ the sigma-algebra generated by all variables except for those associated with the edge $e$.
Since $D_e^xT_t$ is determined by $\omega_t(e')$ for $e'\neq e$, it follows that $Z_t(e)$ is $\mathcal{F}_e$-measurable. From~\eqref{eq:DZ} we obtain that
\begin{equation}\label{eq:inf_cov}
\Inf_e(T_0,T_t)=\E\big[\Cov\big((Z_0(e)-\tilde\omega)_+,(Z_t(e)-\tilde\omega)_+\big|\mathcal{F}_e\big)\big],
\end{equation}
where $\tilde\omega$ is a generic $F$-distributed random variable independent of everything else.

Our goal will be to relate the co-influences to the joint inclusion in the geodesic. Let $\pi_t:=\pi_t(0,v)$. The connection to geodesics can be conceived from~\eqref{eq:inf_cov}, as the event $\{e\in\pi_t\}$ is roughly equivalent with the weight at $e$ not exceeding $Z_t(e)$. In turn, the joint inclusion is connected to the expected overlap through the identity
\begin{equation}\label{eq:overlap_id}
\E[|\pi_0\cap\pi_t|]=\sum_{e\in\Ec}\PP(e\in\pi_0\cap\pi_t).
\end{equation}
Before embarking on the quest to make this precise, we end this section with an observation regarding the expected overlap of the geodesics.

\begin{lemma}\label{lma:overlap_monotonicity}
For every $v\in\Z^d$, the function $\E[|\pi_0\cap\pi_t|]$ is non-increasing in $t$.
\end{lemma}

\begin{proof}
Consider the function $f_e=\mathbf{1}_{\{e\in\pi\}}$. By Lemma~\ref{lma:mono_inf}, we have that
$$
Q_t(f_e)=\E[f_e(\omega_0)f_e(\omega_t)]=\PP(e\in\pi_0\cap\pi_t)
$$
is non-increasing in $t$. Hence, the result follows from~\eqref{eq:overlap_id}.
\end{proof}

We end this section showing how it follows from~\eqref{eq:cov_geos} that the system has a stability/chaos transition at $t\asymp\frac{1}{|v|}\Var(T)$. Apart from the relation itself, we make use of the linearity of geodesics in~\eqref{eq:linear} and monotonicity of the overlap in Lemma~\ref{lma:overlap_monotonicity}.

\begin{prop}\label{prop:transition}
Suppose that $F$ satisfies~\eqref{eq:pc_moment} and that for all $v\in\Z^d$ and $t\in[0,1]$
\begin{equation}\label{eq:assumption}
\E\big[(T_0(0,v)-T_t(0,v))^2\big]\asymp\int_0^t\E\big[\big|\pi_0(0,v)\cap\pi_s(0,v)\big|\big]\,ds.
\end{equation}
Then there exists $C<\infty$ such that for all $v\in\Z^d$ and $0<\alpha<\frac{|v|}{\Var(T(0,v))}$ the following holds.
\begin{enumerate}[\quad (a)]
\item For $t\le\alpha\frac{1}{|v|}\Var(T(0,v))$ we have $\Corr\big(T_0(0,v),T_t(0,v)\big)\ge1-C\alpha$.
\item For $t\ge\alpha\frac{1}{|v|}\Var(T(0,v))$ we have $\E\big[\big|\pi_0(0,v)\cap\pi_t(0,v)\big|\big]\le C\frac{|v|}{\alpha}.$
\end{enumerate}
\end{prop}

\begin{proof}
As above, we write $T_t=T_t(0,v)$ and $\pi_t=\pi_t(0,v)$ for short. We first establish stability for small $t$. Using~\eqref{eq:assumption} and~\eqref{eq:linear} we have, for some $C<\infty$, that
$$
\E\big[(T_0-T_t)^2\big]\le C\int_0^t\E[|\pi_0\cap\pi_s|]\,ds\le Ct\,\E[|\pi_0|]\le C^2t|v|.
$$
Together with~\eqref{eq:expansion} we obtain for $t\le\alpha\frac{1}{|v|}\Var(T)$ that
$$
\Cov(T_0,T_t)=\Var(T)-\frac12\E\big[(T_0-T_t)^2\big]\ge\Var(T)-\frac{C^2}{2}t|v|\ge\Var(T)\Big[1-\frac{C^2}{2}\alpha\Big].
$$

We next establish chaos for large $t$. From~\eqref{eq:assumption} and Lemma~\ref{lma:overlap_monotonicity} there exists $c>0$ such that
$$
\Var(T)\ge c\int_0^t\E[|\pi_0\cap\pi_s|]\,ds\ge ct\,\E[|\pi_0\cap\pi_t|].
$$
For $t\ge\alpha\frac{1}{|v|}\Var(T)$ this gives
$$
\E[|\pi_0\cap\pi_t|]\le\frac1c\frac{|v|}{\alpha},
$$
as required.
\end{proof}

\section{Integer-valued edge weight distributions}\label{sec:integer}

Since the analysis of the co-influences is more straightforward in the integer-valued setting, we consider this setting first. Our goal is thus to prove Theorem~\ref{th:integer}. In view of Propositions~\ref{prop:cov_infinite} and~\ref{prop:transition}, as well as~\eqref{eq:overlap_id}, the key step will be to establish the following.

\begin{prop}\label{prop:int_inf}
Suppose that $F$ is supported on $\{0,1,2,\ldots\}$ and that~\eqref{eq:pc_moment} holds. Then
$$
\Inf_e(T_0,T_t)\asymp\PP(e\in\pi_0\cap \pi_t)\quad\text{for all }t\in[0,1], v\in\Z^d\text{ and }e\in\mathcal{E}.
$$
\end{prop}

\begin{proof}
Since $F$ is supported on the integers, then the infimum $r$ of the support is necessarily an integer and $F(r)=\PP(\omega_e=r)\in(0,1)$. Moreover, since also $T_t$ is integer-valued, it is, crucially, the case that $Z_t(e)$ is integer-valued too, for all $t\in[0,1]$, $v\in\Z^d$ and $e\in\mathcal{E}$.
It follows that, on the event $\{Z_t(e)>r\}$, we have
$$
\int(Z_t(e)-x)_+\,dF(x)\le(Z_t(e)-r)F(r)+(Z_t(e)-r-1)(1-F(r))=(Z_t(e)-r)-(1-F(r)).
$$
Hence, by~\eqref{eq:DZ}, we obtain that
\begin{equation}\label{eq:DT_int_bound}
D_e^rT_t=\int(Z_t(e)-x)_+\,dF(x)-(Z_t(e)-r)\le-(1-F(r))\mathbf{1}_{\{Z_t(e)>r\}},
\end{equation}
since also $D_e^rT_t=0$ on the event $\{Z_t(e)=r\}$.

By Lemma~\ref{lma:mono_inf}, we know that $\E[D_e^xT_0D_e^xT_t]=Q_t(D_e^xT)$ is non-negative for all $x$, which gives the lower bound
$$
\Inf_e(T_0,T_t)=\int\E[D_e^xT_0D_e^xT_t]\,dF(x)\ge F(r)\E[D_e^rT_0D_e^rT_t].
$$
Together with~\eqref{eq:DT_int_bound}, this gives that
$$
\Inf_e(T_0,T_t)\ge F(r)(1-F(r))^2\,\PP\big(Z_0(e)>r,Z_t(e)>r\big).
$$
By definition of $\pi_t$ we have $\{e\in\pi_t\}\subseteq\{Z_t(e)>r\}$, and may thus conclude that
$$
\Inf_e(T_0,T_t)\ge F(r)(1-F(r))^2\,\PP\big(e\in\pi_0\cap\pi_t\big).
$$

To obtain a matching upper bound, we recall~\eqref{eq:DT_bound}, which gives
\begin{equation}\label{eq:inf_int_upper}
\Inf_e(T_0,T_t)\le\int(\mu+x)^2\,dF(x)\cdot\PP(Z_0(e)>r,Z_t(e)>r),
\end{equation}
which is finite since $F$ has finite second moment. Let $A=\{Z_0(e)>r,Z_t(e)>r\}$ and $B=\{\omega_0(e)=\omega_t(e)=r\}$, and note that $A\cap B\subseteq\{e\in\pi_0\cap\pi_t\}$. Moreover, $A$ is $\mathcal{F}_e$-measurable, and hence independent of $B$. Therefore,
$$
\PP(A)=\frac{\PP(A\cap B)}{\PP(B)}\le\frac{\PP(e\in\pi_0\cap\pi_t)}{F(r)^2},
$$
and hence
$$
\Inf_e(T_0,T_t)\le\frac{1}{F(r)^2}\int(\mu+x)^2\,dF(x)\cdot\PP(e\in\pi_0\cap\pi_t),
$$
as required.
\end{proof}

Theorem~\ref{th:integer} is now an easy consequence.

\begin{proof}[Proof of Theorem~\ref{th:integer}]
According to Proposition~\ref{prop:transition} it will suffice to verify that~\eqref{eq:assumption} holds uniformly in $v\in\Z^d$ and $t\in[0,1]$. By Proposition~\ref{prop:cov_infinite} it will suffice to show that
$$
\sum_{e\in\mathcal{E}}\Inf_e(T_0,T_t)\asymp\E[|\pi_0\cap\pi_t|]\quad\text{for }v\in\Z^d\text{ and }t\in[0,1].
$$
However, this is a consequence of Proposition~\ref{prop:int_inf} and~\eqref{eq:overlap_id}.

Assuming~\eqref{eq:var_cond} we have that~\eqref{eq:sublin} holds. For $t\gg1/\log|v|$ we obtain, in particular, that 
$$
\E[|\pi_0\cap\pi_t|]\le C\frac1t\Var(T)\le C\frac{|v|}{t\log|v|}=o(|v|),
$$
as required.
\end{proof}

\section{General edge weight distributions}\label{sec:cont}

We now proceed with the analysis in the more general setting, and prove Theorem~\ref{th:cont}. We shall bound the co-influences from below and above separately, which together amount to a version of~\eqref{eq:inf_asymp}, and use these bounds to prove Theorem~\ref{th:cont} at the end.

Our bounds on the co-influences will be the result of a more detailed analysis of the function $D_e^xT_t$ depicted in Figure~\ref{fig:DTt_YZH}. The additional difficulty comes from adequately estimating the co-influence when $Z_t(e)>r$ is small, and the function $D_e^xT_t$ close to zero, which cannot happen for integer-valued weight distributions.
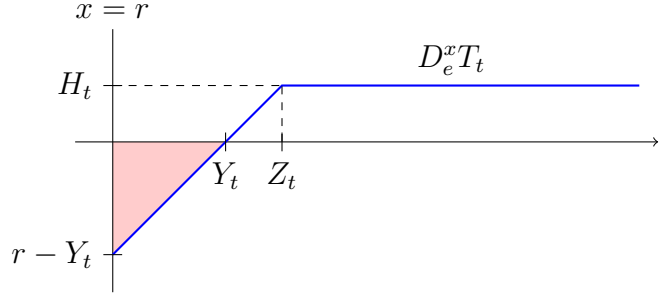
\begin{figure}[htbp]
\begin{center}
\begin{tikzpicture}[scale=.5]
\fill[fill=red!20!white] (0,0) -- (3,0) -- (0,-3) -- cycle;
\draw (0,-4) -- (0,3);
\draw[->] (-1,0) -- (14.5,0);
\draw[draw=blue,thick] (0,-3) -- (4.5,1.5) -- (14,1.5);
\draw (4.5,-.25) -- (4.5,.25);
\draw (3,-.25) -- (3,.25);
\draw (-.25,-3) -- (.25,-3);
\draw (-.25,1.5) -- (.25,1.5);
\draw[dashed] (4.5,0) -- (4.5,1.5);
\draw[dashed] (0,1.5) -- (4.5,1.5);
\draw (4.5,-.25) node[anchor=north] {$Z_t$};
\draw (9,1.5) node[anchor=south] {$D_e^xT_t$};
\draw (0,3) node[anchor=south] {$x=r$};
\draw (-.25,-3) node[anchor=east] {$r-Y_t$};
\draw (-.25,1.5) node[anchor=east] {$H_t$};
\draw (3,-.25) node[anchor=north] {$Y_t$};
\end{tikzpicture}
\end{center}
\caption{\small Illustration of the variables $Z_t=Z_t(e)$, $Y_t=Y_t(e)$ and $H_t=H_t(e)$.}
\label{fig:DTt_YZH}
\end{figure}
In addition to $Z_t=Z_t(e)$, we identify two other variables associated with the function $D_e^xT_t$, which are illustrated in Figure~\ref{fig:DTt_YZH}. Since $D_e^xT_t$ is monotone and integrates to zero, there is a first point at which it intersects the real axis, and remains non-negative after. We denote this point by $Y_t=Y_t(e)$, i.e.\
\begin{equation*}
Y_t(e):=\min\{x\ge r:D_e^xT_t\ge0\}.
\end{equation*}
Of course, $r\le Y_t(e)\le Z_t(e)$ and $Y_t(e)=Z_t(e)=r$ if $D_e^xT_t$ is constant (and hence equal to zero). In addition, we let $H_t=H_t(e)$ denote the height of the flat segment of $D_e^xT_t$, which is attained for $x\ge Z_t(e)$, i.e.\
$$
H_t(e):=D_e^{Z_t(e)}T_t.
$$
Recall that $Z_t(e)$ is finite, so that $H_t(e)$ is well-defined. However, we remark that $Z_t(e)$ is not necessarily in the support of $F$, and the value $H_t(e)$ is not necessarily attained for $x$ in the support of $F$.

When it is clear from the context which edge $e$ is referred to, we suppress the dependence on $e$ and simply write $Y_t$, $Z_t$ and $H_t$ for compactness.

\subsection{Co-influence lower bound}

We first establish the following lower bound on the co-influences.

\begin{prop}\label{prop:inf_lower_bound}
Suppose that $F$ has finite second moment and that $F(r)=0$. For every $\vep>0$ there exists $C(\vep)>0$ such that for all $e\in\Ec$ and $t\in[0,1]$ we have
$$
\Inf_e(T_0,T_t)\ge C(\vep)\big[\PP(e\in\pi_0\cap\pi_t)-2\PP\big(\{e\in\pi_0\}\cap\{\omega_0(e)\le r+\vep\}\big)\big].
$$
\end{prop}

We will derive the lower bound in three steps. We first show that it will suffice to estimate the contribution to the co-influences that comes from the shaded region below the axis in Figure~\ref{fig:DTt_YZH}. Second, we show that if $Z_t(e)$ is not too small, then (the contribution from) the shaded region is also not too small. This will lead to a lower bound of the form
$$
\Inf_e(T_0,T_t)\ge C(\vep)\PP\big(\{Z_0(e)>r+\vep\}\cap \{Z_t(e)>r+\vep\}\big).
$$
Finally, we use that if $e\in\pi_t$, then either $Z_t(e)>r+\vep$, or $e\in\pi_t$ and $\omega_t(e)\le r+\vep$.

In the first step we use the monotonicity of $D_e^xT_t$ to obtain the following lemma.

\begin{lemma}\label{le:positive_part}
For every $e\in\Ec$ and $t\in[0,1]$, we have
$$
\int\Do\Dt\,dF(x)\geq \int(\Do)_-(\Dt)_-\,dF(x).
$$
\end{lemma}

\begin{proof}
Let $e\in\Ec$ and $t\in[0,1]$ be fixed. By symmetry, we may without loss of generality assume that $Y_0\leq Y_t$. By~\eqref{eq:D_balance}, it holds that $\Dt$ integrates to zero, and hence, since $D_e^xT_t$ is negative below $Y_t$ and positive above, that
\begin{equation}\label{eq:vv0}
-\int_{[r,Y_t]}\Dt\,dF(x)=\int_{[Y_t,\infty)}\Dt\,dF(x).
\end{equation}
(At $Y_t$ the integrand is zero, so it does not matter if we may include the point $Y_t$ in the interval of integration or not.) Since $Y_0\le Y_t$ by assumption and $D_e^xT_t$ is negative below $Y_t$, it follows that
$$
-\int_{[Y_0,Y_t]}\Dt\,dF(x)\leq \int_{[Y_t,\infty)}\Dt\,dF(x).
$$
Using the fact that $\Do$ is non-decreasing and non-negative on $[Y_0,\infty)$, we obtain that
$$
-\int_{[Y_0,Y_t]}\Do\Dt\,dF(x)\leq \int_{[Y_t,\infty)} \Do\Dt\,dF(x).
$$
Since both $\Do$ and $\Dt$ are negative on $[r,Y_0)$ and $\Do$ is non-negative on $[Y_0,\infty)$, we conclude from the above that
$$
\int_{[r,\infty)} \Do\Dt\,dF(x)\geq \int_{[r,Y_0)}\Do\Dt\,dF(x)=\int_{[r,\infty)} (\Do)_-(\Dt)_-\,dF(x),
$$
as required.
\end{proof}

In the second step we show that, if $Y_t(e)$ is close to $r$, then $Z_t(e)$ is close to $r$ too.

\begin{lemma}\label{le:delta_bd}
For every $e\in\Ec$, $t\in[0,1]$ and $\vep>0$ such that $F(r+\vep)\leq 1/2$, we have that if $Z_t(e)\ge r+\vep$, then $Y_t(e)\ge r+ \vep/2$.
\end{lemma}

\begin{proof}
Let $e\in\Ec$ and $t\in[0,1]$ be fixed. Fix $\vep>0$ such that $F(r+\vep)\le1/2$, and suppose that $Z_t-r\ge\vep$. By~\eqref{eq:vv0}, and since $-\Dt\le Y_t-r$ on $[r,Y_t]$, we have
\begin{equation}\label{lhs}
\int_{[Y_t,\infty)}\Dt\,dF(x)=-\int_{[r,Y_t]}\Dt\,dF(x)\le
(Y_t-r)F(Y_t).
\end{equation}
Now suppose that $Y_t<r+\vep/2$. We then have $D_e^{r+\vep/2}T_t\ge0$ and, since $Z_t\ge r+\vep$, hence that $D_e^xT_t\ge\vep/2$ on $[r+\vep,\infty)$. By choice of $\vep>0$, this gives
\begin{equation}\label{rhs}
\int_{[Y_t,\infty)} \Dt\,dF(x)\ge \frac{\vep}{2}\big(1-F(r+\vep)\big)\geq \frac{\vep}{4}.
\end{equation}
Moreover, if $Y_t<r+\vep/2$, then $F(Y_t)\le F(r+\vep)\le1/2$, so~\eqref{lhs} and~\eqref{rhs} yield that $Y_t-r\ge\vep/2$.
\end{proof}

With these lemmas at hand we prove Proposition~\ref{prop:inf_lower_bound}.

\begin{proof}[Proof of Proposition~\ref{prop:inf_lower_bound}]
Let $e\in\Ec$ and $t\in[0,1]$ be fixed.
It will suffice to consider $\vep>0$ such that $F(r+\vep)\le1/2$, so we fix $\vep$ accordingly.
By Lemma \ref{le:positive_part}, the influence can be lower bounded as
$$
\Inf_e(T_0,T_t)=\int\E[\Dt\Do]\,dF(x)\ge\int \E[(\Do)_-(\Dt)_-]\,dF(x).
$$
Let $\Delta:=\min\{Y_0(e)-r,Y_t(e)-r\}$, and note that
$$
\int(\Do)_-(\Dt)_-\,dF(x)\ge \Big(\frac{\Delta}{2}\Big)^2F(r+\Delta/2).
$$
Furthermore, by Lemma \ref{le:delta_bd}, we have on the event ${\{Z_0(e)>r+\vep\}\cap \{Z_t(e)>r+\vep\}}$ that $\Delta\geq \vep/2$, and so it follows that
\begin{equation}\label{eq:infZ_bound}
\Inf_e(T_0,T_t)\ge\Big(\frac{\vep}{4}\Big)^2F(r+\vep/4)\PP\big(\{Z_0(e)>r+\vep\}\cap \{Z_t(e)>r+\vep\}\big).
\end{equation}
Now observe that
\begin{align*}
\{e\in\pi_t\} & \subseteq \{Z_t(e)>r+\vep\}\cup\big[\{e\in\pi_t\}\cap \{Z_t(e)\leq r+\vep\}\big]\\
 & \subseteq \{Z_t(e)>r+\vep\}\cup\big[\{e\in\pi_t\}\cap\{\omega_t(e)\le r+\vep\}\big],
\end{align*}
so that
$$
\{Z_t(e)>r+\vep\}\supseteq \{e\in\pi_t\}\setminus \big[\{e\in\pi_t\}\cap\{\omega_t(e)\leq r+\vep\}\big].
$$
This means that we can bound
$$
\PP\big(\{Z_0(e)>r+\vep\}\cap \{Z_t(e)>r+\vep\}\big)\ge \PP(e\in\pi_0\cap \pi_t)-2\PP\big(\{e\in\pi_0\}\cap\{\omega_0(e)\leq r+\vep\}\big),
$$
which together with~\eqref{eq:infZ_bound} completes the proof.
\end{proof}

\subsection{Co-influence upper bound}

We proceed with the upper bound on the co-influences. The argument requires only~\eqref{eq:pc_moment}, and the next proposition is thus stated accordingly.

\begin{prop}\label{prop:cont_upper}
Suppose that $F$ satisfies~\eqref{eq:pc_moment}. There exists $C<\infty$ such that for every $e\in\Ec$ and $t\in[0,1)$ we have that
$$
\Inf_e(T_0,T_t)\le \frac{C}{1-t}\PP(e\in\pi_0\cap\pi_t).
$$
\end{prop}

In a first step we bound the function $D_e^xT_t$ when $Z_t(e)$ is small.

\begin{lemma}\label{lma:H} Suppose that $F(r)<1/2$. There exists $\gamma>r$ such that, for every $e\in\Ec$ and $t\in[0,1]$, we have that if $Z_t(e)\le \gamma$ then $H_t(e)\le F(Y_t(e))$.
\end{lemma}

\begin{proof}
Let $e\in\Ec$ and $t\in[0,1]$ be fixed. Fix $\gamma\in(r,r+1/2)$ such that $F(\gamma)\le1/2$, and suppose that $Z_t\le \gamma$. Then
$$
-\int_{[r,Y_t]}\Dt\,dF(x)\le(Y_t-r)F(Y_t)\le F(Y_t)/2.
$$
Furthermore, the balance equation in~\eqref{eq:vv0} and the assumption that $Z_t\le\gamma$ imply that
$$
-\int_{[r,Y_t]}\Dt\,dF(x)\ge \int_{(Z_t,\infty)} \Dt\,dF(x) \ge H_t(1-F(Z_t))\ge H_t/2.
$$
(The mid-inequality is an equality unless $F$ has an atom at $Z_t$.) The two equations give $H_t\le F(Y_t)$, as required.
\end{proof}

\begin{proof}[Proof of Proposition \ref{prop:cont_upper}]
If $F(r)>0$, then the argument in the integer-valued setting applies. We may therefore assume that $F(r)=0$.

Fix $e\in\Ec$ and $t\in[0,1)$. Also fix $\gamma>r$ as in Lemma~\ref{lma:H}. We shall bound the co-influence separately depending on whether both $Z_0$ and $Z_t$ exceed $\gamma$ or not, by dividing the co-influence $\Inf_e(T_0,T_t)$ into the following sum
\begin{equation}\label{eq:inf_split}
\int\E[D_e^xT_0D_e^xT_t\mathbf{1}_{\{Z_0(e)>\gamma,Z_t(e)>\gamma\}}]\,dF(x)+\int\E[D_e^xT_0D_e^xT_t\mathbf{1}_{\{Z_0(e)\le\gamma\}\cup\{Z_t(e)\le\gamma\}}]\,dF(x).
\end{equation}
The former of the two terms above can be bounded in the same way as in the integer-valued setting. More precisely, appealing to~\eqref{eq:DT_bound} yields
\begin{equation*}
\int\E[D_e^xT_0D_e^xT_t\mathbf{1}_{\{Z_0(e)>\gamma,Z_t(e)>\gamma\}}]\,dF(x)\le\int(\mu+x)^2\,dF(x)\,\PP(Z_0(e)>\gamma,Z_t(e)>\gamma).
\end{equation*}
Now, with $A=\{Z_0(e)>\gamma,Z_t(e)>\gamma\}$ and $B=\{\omega_0(e)\le\gamma,\omega_t(e)\le\gamma\}$, we have $A\cap B\subseteq\{e\in\pi_0\cap\pi_t\}$. Since $A$ and $B$ are independent, and since $F$ has finite second moment, we obtain that
\begin{equation}\label{eq:inf_split_bound}
\int\E[D_e^xT_0D_e^xT_t\mathbf{1}_{\{Z_0(e)>\gamma,Z_t(e)>\gamma\}}]\,dF(x)\le C\frac{\PP(e\in\pi_0\cap\pi_t)}{F(\gamma)^2}
\end{equation}
for some finite constant $C$.

We proceed with the bound on the latter term of~\eqref{eq:inf_split}. Suppose first that $Y_0\le Y_t$, which also implies that $Z_0\le Z_t$. Then, on the interval $[r,Y_0]$ both $D_e^xT_0$ and $D_e^xT_t$ are negative and by~\eqref{eq:DTmu} lower bounded by $-\mu$, whereas on $[Y_t,\infty)$ both $D_e^xT_0$ and $D_e^xT_t$ are positive and by definition and~\eqref{eq:DT_bound} upper bounded by $H_0$ and $\mu+x$, respectively. Since for the remaining values of $x$ the functions have different signs, we obtain the upper bound
$$
\int D_e^xT_0D_e^xT_t\,dF(x)\le\mu^2F(Y_0)+H_0\int_{[Y_t,\infty)}(\mu+x)\,dF(x)\le \mu^2F(Y_0)+2\mu H_0.
$$
By assumption $Z_0\le Z_t$ and hence, on the event that either $Z_0\le\gamma$ or $Z_t\le\gamma$, we have that $Z_0\le\gamma$. On this event we have by Lemma~\ref{lma:H} that $H_0\le F(Y_0)$, which gives the further upper bound $(\mu^2+2\mu)F(Y_0)$. If instead $Z_t\le Z_0$ and either $Z_0<\gamma$ or $Z_t<\gamma$, we obtain analogously the upper bound $(\mu^2+2\mu)F(Y_t)$. This shows that
$$
\int D_e^xT_0D_e^xT_t\,dF(x)\mathbf{1}_{\{Z_0(e)\le\gamma\}\cup\{Z_t(e)\le\gamma\}}\le(\mu^2+2\mu)F(\overline Y(e)).
$$
where $\overline Y(e)=\min\{Y_0(e),Y_t(e)\}$.

Next, we note that for $y\ge0$ and $t\in[0,1)$ we have
$$
\PP\big(\omega_0(e)\le y,\omega_t(e)\le y\big)\ge\PP\big(\omega_0\le y\big| U(e)>t\big)\PP(U(e)>t)=(1-t)F(y),
$$
and hence that
$$
F(\overline Y(e))\le\frac{1}{1-t}\PP\big(\omega_0(e)\le \overline Y(e),\omega_t(e)\le\overline Y(e)\big|\mathcal{F}_e\big).
$$
Now either $\overline Y(e)=0$ or $\overline Y(e)<\min\{Z_0(e),Z_t(e)\}$, which gives the further upper bound
$$
F(\overline Y(e))\le \frac{1}{1-t}\PP\big(\omega_0(e)<Z_0(e),\omega_t(e)<Z_t(e)\big|\mathcal{F}_e\big)\le\frac{1}{1-t}\PP(e\in\pi_0\cap\pi_t|\mathcal{F}_e).
$$
Taking expectation results in the bound
$$
\int\E[D_e^xT_0D_e^xT_t\mathbf{1}_{\{Z_0(e)\le\gamma\}\cup\{Z_t(e)\le\gamma\}}]\,dF(x)\le\frac{\mu^2+2\mu}{1-t}\PP(e\in\pi_0\cap\pi_t),
$$
which together with~\eqref{eq:inf_split_bound} completes the proof.
\end{proof}

\subsection{Proof of Theorem~\ref{th:cont}}

We finally combine the upper and lower bounds on the co-influences to prove Theorem~\ref{th:cont}. For the lower bound we borrow a lemma from the theory of greedy lattice animals, in order to obtain a bound on the number of edges with low weight contained in a geodesic. The lemma stated below is proved in~\cite[Proposition~1.3]{damhanjanlamshe}, and is a strengthening of a result from~\cite[Theorem~6.6]{damhansos15}.

\begin{lemma}[{\cite[Proposition~1.3]{damhanjanlamshe}}]\label{lma:small_weights}
Suppose that $F(0)<p_c$. Then there exists a constant $C<\infty$ such that for all $v\in\Z^d$ and $\vep>0$
$$
\E\big[\#\big\{e\in\pi(0,v):\omega(e)\le r+\vep\big\}\big]\le C|v|F(r+\vep)^{1/d}.
$$
\end{lemma}

\begin{proof}[Proof of Theorem~\ref{th:cont}.]
Suppose that $F$ has finite second moment and that $F(r)=0$.
From Proposition~\ref{prop:cov_infinite} and the monotonicity of the co-influences, in Lemma~\ref{lma:inf_inf}, we obtain
$$
\Var(T)=\int_0^1\sum_{e\in\Ec}\Inf_e(T_0,T_t)\,dt\le 2\int_0^{1/2}\sum_{e\in\Ec}\Inf_e(T_0,T_t)\,dt.
$$
Furthermore, by Proposition~\ref{prop:cont_upper}, there exists a universal constant $C<\infty$ such that $\Inf_e(T_0,T_t)\le 2C\,\PP(e\in\pi_0\cap\pi_t)$ for $t\in[0,1/2]$. Together with~\eqref{eq:overlap_id} this gives
\begin{equation}\label{eq:cont_upper}
\mbox{Var}(T)\le 4C\int_0^{1/2}\E[|\pi_0\cap\pi_t|]\,dt\le 4C\int_0^1\E[|\pi_0\cap\pi_t|]\,dt,
\end{equation}
which provides the upper bound on the variance.

Combining~\eqref{eq:expansion}, Proposition~\ref{prop:cont_upper} and~\eqref{eq:linear} we obtain for $t\le1/2$ that
$$
\Cov(T_0,T_t)=\Var(T)-2C\int_0^t\E[|\pi_0\cap\pi_s|]\,ds\ge\Var(T)-2Ct|v|.
$$
Consequently, for $t\ll\frac{1}{|v|}\Var(T)$ we find that $\Corr(T_0,T_t)\ge1-o(1)$ as $|v|\to\infty$.

Moving to the lower bound, we have by Proposition~\ref{prop:cov_infinite} that
$$
\Var(T)\ge\int_0^t\sum_{e\in\Ec}\Inf_e(T_0,T_s)\,ds
$$
for all $t\in[0,1]$. From Proposition~\ref{prop:inf_lower_bound} and~\eqref{eq:overlap_id} we obtain for every $\vep>0$ a constant $C(\vep)>0$ such that for all $t\in[0,1]$
\begin{equation}\label{eq:cont_lower}
\begin{aligned}
\Var(T)&\ge C(\vep)\int_0^t\sum_{e\in\Ec}\big[\PP(e\in\pi_0\cap\pi_s)-2\PP\big(\{e\in\pi_0\}\cap\{\omega_0(e)\le r+\vep\}\big)\big]\,ds\\
&= C(\vep)\bigg(\int_0^t\E[|\pi_0\cap\pi_s|]\,ds-2t\,\E\big[\#\big\{e\in\pi_0:\omega_0(e)\le r+\vep\big\}\big]\bigg).
\end{aligned}
\end{equation}
Taking $t=1$ we obtain the claimed lower bound from Lemma~\ref{lma:small_weights}.

Recall from Lemma~\ref{lma:overlap_monotonicity} that $\E[|\pi_0\cap\pi_t|]$ is monotone (non-increasing) as a function of $t$. It follows from~\eqref{eq:cont_lower} and Lemma~\ref{lma:small_weights} that for every $\vep>0$ there exists $C(\vep)>0$ such that for all $t\in(0,1)$ we have
\begin{equation}
\label{eq:overlapcont}
\E[|\pi_0\cap\pi_t|]\leq\frac1t\int_0^t\E[|\pi_0\cap\pi_s|]\,ds\leq C(\vep)^{-1}\frac{\Var(T)}{t}+2C|v|F(r+\vep)^{1/d},
\end{equation}
where $C$ is a fixed constant, not depending on $\vep$. Since $F(r)=0$, we may for every $\delta>0$ choose $\vep>0$ small so that 
$$
\E[|\pi_0\cap\pi_t|]\le C(\vep)^{-1}\frac{\Var(T)}{t}+\delta|v|.
$$
Assuming that $t\gg\frac{1}{|v|}\Var(T)$ gives $E[|\pi_0\cap\pi_t|]\le2\delta|v|$ for all but finitely many $v\in\Z^d$. Since $\delta>0$ was arbitrary, this establishes chaos for $t\gg\frac{1}{|v|}\Var(T)$.

If, in addition,~\eqref{eq:var_cond} holds, then $\Var(T)\le C|v|/\log|v|$ by~\eqref{eq:sublin}. The above then implies sublinear growth of the overlap, as $|v|\to\infty$, for $t\gg1/\log|v|$.
\end{proof}

\section{Multiple valleys}\label{sec:mv}

In this section we prove Theorem~\ref{th:mv}.  Suppose that $F$ is continuous and satisfies~\eqref{eq:var_cond}. For $i = 1,2, \dots, k$ let $\{ \omega^{\sss (i)}(e)\}_{e \in \Ec}$ independent families of i.i.d.\ edge weights with common distribution $F$. In addition, for $i=1, 2,\dots, k$, let $\{U^{\sss (i)}(e)\}_{e \in \Ec}$ be independent configurations of i.i.d.\ uniformly distributed random variables on the interval $[0,1]$. Starting from the usual weight configuration $\omega=\{\omega(e)\}_{e \in \Ec}$ at time $t=0$, we define independent perturbations of this weight configuration as
\begin{equation}
\omega_t^{\sss (i)}(e) = \begin{cases}
\omega(e) & \text{ if } U^{\sss (i)}(e)>t, \\
\omega^{\sss (i)}(e) & \text{ if } U^{\sss (i)}(e)\le t.
\end{cases}
\end{equation}
That is, in the resulting configuration $\omega_t^{\sss (i)}=\{\omega_t^{\sss (i)}(e)\}_{e\in\Ec}$ each edge has been resampled with probability $t$, and this is carried out independently for the $k$ configurations.

We note that the joint distribution of the pair $(\omega_t^{\sss (i)},\omega_t^{\sss (j)})$, for $i\neq j$, is equal to the joint distribution of the pair of configurations $(\omega_0,\omega_s)$, where $s=2t-t^2$. This is because, independently for each edge $e$, we will have $\omega_t^{\sss (i)}(e)=\omega_t^{\sss (j)}(e)=\omega(e)$, unless $\omega(e)$ has been replaced by an independent variable in one of two independent attempts, each having success probability $t$. Hence the probability of resampling equals $s=2t-t^2$.

Write $T_t^{\sss (i)}=T_t^{\sss (i)}(0,v)$ for the first passage time from $0$ to $v$ in the configuration $\{\omega_t^{\sss (i)}(e)\}_{e \in \Ec}$ and write $\pi_t^{\sss (i)}=\pi_t^{\sss (i)}(0,v)$ for the corresponding geodesic, which is almost surely uniquely defined, since we assume that $F$ is continuous. Similarly, write $T=T(0,v)$ and $\pi=\pi(0,v)$ for the passage time and geodesic between $0$ and $v$ in the initial configuration $\{\omega(e)\}_{e \in \Ec}$, and note that $T_0^{\sss (i)}=T$ and $\pi_0^{\sss (i)}=\pi$ for all $i$. Finally, recall that $T(\Gamma)$ denotes the weight-sum along the path $\Gamma$, and hence that $T(\pi_t^{\sss (i)})$ denotes the travel time of the path $\pi_t^{\sss (i)}$, in the initial configuration $\{\omega(e)\}_{e \in\Ec}$.

We introduce a notation for the {\bf maximum overlap} between the paths $\pi$ and $\pi_t^{\sss (i)}$, for $i=1,2,\ldots,k$, as
$$
O_k := \max \Big\{\max_{\substack{i,j=1, \ldots, k \\i \neq j}} \left|\pi_t^{\sss (i)} \cap \pi_t^{\sss (j)}\right|, \max_{i=1, \ldots, k} \left|\pi \cap \pi_t^{\sss (i)}\right|\Big\},
$$
and the {\bf maximum time difference} of the same paths as
$$
{\Delta T}_k := \max_{i=1, \ldots, k} \big[T(\pi_t^{\sss (i)}) - T\big].
$$

In the remainder of this section our goal will be to show that, as $|v|\to\infty$, we may choose $t=t(v)\to0$ and $k=k(v)\to\infty$ in such a way that $O_k=o(|v|)$ and $\Delta T_k=o(|v|)$ with probability tending to 1. This will complete the proof of Theorem~\ref{th:mv}. \vspace{0.1cm}

\noindent
{\bf Maximum overlap.} It follows from~\eqref{eq:sublin} and~\eqref{eq:overlapcont} that for every $\vep>0$ there exists a constant $C(\vep)>0$ such that for all $t>0$ and $v\in\Z^d$ we have
\begin{equation}\label{eq:over_bound}
\E\big[\big|\pi \cap \pi_t^{\sss (i)}\big|\big] \leq  \frac{|v|}{C(\vep) t \log |v|} + C'\frac{|v| F(r+\vep)^{1/d}}{t},
\end{equation}
where $C'$ is a universal constant. Denote the right-hand side in~\eqref{eq:over_bound} by $\psi(v,\vep,t)$, and note that since $s\ge t$ we obtain for $i\neq j$ that also $\E[|\pi_t^{\sss (i)}\cap\pi_t^{\sss (j)}|]\le \psi(v,\vep,t)$, and hence that
\begin{equation}\label{eq:EOk}
\E[O_k]\le\psi(v,\vep,t)(k+1)^2.
\end{equation}
Set $t=t(\vep)=F(r+\vep)^{1/2d}$ so that $t(\vep)\to0$ as $\vep\to0$. Since $t(\vep)>0$ and $C(\vep)>0$ for all $\vep>0$, we may choose $\vep=\vep(v)\to0$ such that $t(\vep)C(\vep)\sqrt{\log|v|}\ge1$ as $|v|\to\infty$. Then
$$
\psi(v):=\psi\big(v,\vep(v),t(\vep(v))\big)=o(|v|).
$$
For $k\le k^*(v):=(|v|/\psi(v))^{1/4}-1$ it follows from~\eqref{eq:EOk} that $\E[O_k]\le\sqrt{|v|\psi(v)}$. So, with $\alpha(v):=|v|^{3/4}\psi(v)^{1/4}$ we have $\alpha(v)=o(|v|)$ and Markov's inequality gives
\begin{equation}\label{eq:Ok}
\PP\big(O_k>\alpha(v)\big)\le\frac{\sqrt{|v|\psi(v)}}{|v|^{3/4}\psi(v)^{1/4}}=o(1),
\end{equation}
uniformly in $k\le k^*(v)$. \vspace{0.1cm}

\noindent
{\bf Maximum time difference.} Since $T$ and $T_t^{\sss (i)}$ are equal in distribution, it follows that
$$
\E\big[T(\pi_t^{\sss (i)})-T\big]=\E\big[T(\pi_t^{\sss (i)})-T_t^{\sss (i)}\big]\le\E\Bigg[\sum_{e\in\pi_t^{\sss (i)}}\big|\omega(e)-\omega_t^{\sss (i)}(e)\big|\Bigg].
$$
Using that
$$
\big|\omega(e)-\omega_t^{\sss (i)}(e)\big|\le\big(\omega(e)+\omega^{\sss (i)}(e)\big){\bf 1}_{\{U^{\sss (i)}(e)\le t\}},
$$
that $\pi$ and $\pi_t^{\sss (i)}$ are equal in distribution, and that $\pi$ is a function of $\omega$, we find through conditioning on $\pi$, due to the independence of $\omega(e)$, $\omega^{\sss (i)}(e)$ and $U^{\sss (i)}(e)$, that
$$
\E\big[T(\pi_t^{\sss (i)})-T\big]\le\E\bigg[\sum_{e\in\pi}\big(\omega(e)+\omega^{\sss (i)}(e)\big){\bf 1}_{\{U^{\sss (i)}(e)\le t\}}\bigg]=t\,\E[T]+t\,\E[\omega^{\sss (i)}(e)]\E[|\pi|].
$$
From subadditivity and by~\eqref{eq:linear} we conclude that this is at most $tC|v|$ for some constant $C$, and hence that
$$
\E[\Delta T_k]\le ktC|v|.
$$
Now, let $t=t(v)=t(\vep(v))$ and $k^*(v)$ be defined as above, so that $t(v)\to0$ and $k^*(v)\to\infty$ as $|v|\to\infty$. Set $k(v):=\lfloor\min\{t(v)^{-1/2},k^*(v)\}\rfloor$ and $\beta(v):=t(v)^{1/4}|v|$, so that also $k(v)\to\infty$ and $\beta(v)=o(|v|)$ as $|v|\to\infty$. Then, Markov's inequality gives
\begin{equation}\label{eq:DTk}
\PP\big(\Delta T_{k}>\beta(v)\big)\le C\frac{k(v)t(v)|v|}{t(v)^{1/4}|v|}\le Ct(v)^{1/4}=o(1),
\end{equation}
uniformly in $k\le k(v)$. \vspace{0.1cm}

To conclude, it follows from~\eqref{eq:Ok} and~\eqref{eq:DTk} that there exists a sequence $(k(v))_{v\in\Z^d}$ such that $k(v)\to\infty$ as $|v|\to\infty$ and so that with high probability $O_{k(v)}=o(|v|)$ and $\Delta T_{k(v)}=o(|v|)$. This completes the proof of Theorem~\ref{th:mv}.

\bigskip
{\bf Acknowledgement:} The authors are grateful for the comments received from an anonymous referee which helped clarify some parts of the arguments. This work was partially supported by the Swedish Research Council (VR) through the grant numbers 2021-03964 (DA) and 2020-04479 (MD and MS).


\begin{thebibliography}{10}

\bibitem{ahldeisfr24}
D.~Ahlberg, M.~Deijfen, and M.~Sfragara.
\newblock From stability to chaos in last-passage percolation.
\newblock {\em Bull. Lond. Math. Soc.}, 56(1):411--422, 2024.

\bibitem{ahlhof}
D.~Ahlberg and C.~Hoffman.
\newblock Random coalescing geodesics in first-passage percolation.
\newblock Preprint, see \emph{arXiv:\allowbreak 1609.02447}.

\bibitem{arghan20}
L.-P. Arguin and J.~Hanson.
\newblock On absence of disorder chaos for spin glasses on {$\Bbb Z^d$}.
\newblock {\em Electron. Commun. Probab.}, 25:Paper No. 32, 12, 2020.

\bibitem{aufche16}
A.~Auffinger and W.-K. Chen.
\newblock Universality of chaos and ultrametricity in mixed {$p$}-spin models.
\newblock {\em Comm. Pure Appl. Math.}, 69(11):2107--2130, 2016.

\bibitem{aufdamhan17}
A.~Auffinger, M.~Damron, and J.~Hanson.
\newblock {\em 50 years of first-passage percolation}, volume~68 of {\em
  University Lecture Series}.
\newblock American Mathematical Society, Providence, RI, 2017.

\bibitem{benros}
M.~Bena{\"i}m and R.~Rossignol.
\newblock A modified {P}oincar{\'e} inequality and its application to first
  passage percolation.
\newblock Unpublished, see \emph{arXiv:\allowbreak 0602496}.

\bibitem{benros08}
M.~Bena{\"{\i}}m and R.~Rossignol.
\newblock Exponential concentration for first passage percolation through
  modified {P}oincar\'e inequalities.
\newblock {\em Ann. Inst. Henri Poincar\'e Probab. Stat.}, 44(3):544--573,
  2008.

\bibitem{benkalsch99}
I.~Benjamini, G.~Kalai, and O.~Schramm.
\newblock Noise sensitivity of {B}oolean functions and applications to
  percolation.
\newblock {\em Inst. Hautes \'Etudes Sci. Publ. Math.}, 90:5--43, 1999.

\bibitem{benkalsch03}
I.~Benjamini, G.~Kalai, and O.~Schramm.
\newblock First passage percolation has sublinear distance variance.
\newblock {\em Ann. Probab.}, 31(4):1970--1978, 2003.

\bibitem{borlugzhi20}
C.~Bordenave, G.~Lugosi, and N.~Zhivotovskiy.
\newblock Noise sensitivity of the top eigenvector of a {W}igner matrix.
\newblock {\em Probab. Theory Related Fields}, 177(3-4):1103--1135, 2020.

\bibitem{boulugmas13}
S.~Boucheron, G.~Lugosi, and P.~Massart.
\newblock {\em Concentration inequalities}.
\newblock Oxford University Press, Oxford, 2013.
\newblock A nonasymptotic theory of independence, With a foreword by Michel
  Ledoux.

\bibitem{bramoo87}
A.~J. Bray and M.~A. Moore.
\newblock Chaotic nature of the spin-glass phase.
\newblock {\em Phys. Rev. Lett.}, 58(1):57, 1987.

\bibitem{chatterjee1}
S.~Chatterjee.
\newblock Chaos, concentration and multiple valleys.
\newblock Unpublished, see \emph{arXiv:\allowbreak 0810.4221}.

\bibitem{chatterjee2}
S.~Chatterjee.
\newblock Disorder chaos and multiple valleys in spin glasses.
\newblock Unpublished, see \emph{arXiv:\allowbreak 0907.3381}.

\bibitem{chatterjee3}
S.~Chatterjee.
\newblock Spin glass phase at zero temperature in the {E}dwards-{A}nderson
  model.
\newblock Preprint, see \emph{arXiv:\allowbreak 2301.04112}.

\bibitem{chatterjee05}
S.~Chatterjee.
\newblock {\em Concentration inequalities with exchangeable pairs}.
\newblock PhD thesis, Stanford University, 2005.
\newblock UMI number: 3171763.

\bibitem{chatterjee14a}
S.~Chatterjee.
\newblock {\em Superconcentration and related topics}.
\newblock Springer Monographs in Mathematics. Springer, Cham, 2014.

\bibitem{chen13}
W.-K. Chen.
\newblock Disorder chaos in the {S}herrington-{K}irkpatrick model with external
  field.
\newblock {\em Ann. Probab.}, 41(5):3345--3391, 2013.

\bibitem{damhanjanlamshe}
M.~Damron, J.~Hanson, C.~Janjigian, W.-K. Lam, and X.~Shen.
\newblock Tail bounds for the averaged empirical distribution on a geodesic in
  first-passage percolation.
\newblock Preprint, see \emph{arXiv:\allowbreak 2010.08072}.

\bibitem{damhansos15}
M.~Damron, J.~Hanson, and P.~Sosoe.
\newblock Sublinear variance in first-passage percolation for general
  distributions.
\newblock {\em Probab. Theory Related Fields}, 163(1-2):223--258, 2015.

\bibitem{demelbpel}
B.~Dembin, D.~Elboim, and R.~Peled.
\newblock Coalescence of geodesics and the bks midpoint problem in planar
  first-passage percolation.
\newblock \emph{Geom.\ Func.\ Analysis}, to appear.

\bibitem{demgar}
B.~Dembin and C.~Garban.
\newblock Superconcentration for minimal surfaces in first passage percolation
  and disordered Ising ferromagnets.
\newblock Preprint, see \emph{arXiv:\allowbreak 2301.11248}.

\bibitem{eldan20}
R.~Eldan.
\newblock A simple approach to chaos for {$p$}-spin models.
\newblock {\em J. Stat. Phys.}, 181(4):1266--1276, 2020.

\bibitem{fishus86}
D.~S. Fisher and D.~A. Huse.
\newblock Ordered phase of short-range ising spin-glasses.
\newblock {\em Phys. Rev. Lett.}, 56(15):1601, 1986.

\bibitem{galpet}
P.~Galicza and G.~Pete.
\newblock Sparse reconstruction in spin systems {I}: iid spins.
\newblock \emph{Israel Journal of Mathematics}, to appear.

\bibitem{ganham23}
S.~Ganguly and A.~Hammond.
\newblock The geometry of near ground states in {G}aussian polymer models.
\newblock {\em Electron. J. Probab.}, 28:Paper No. 60, 80, 2023.

\bibitem{ganham24}
S.~Ganguly and A.~Hammond.
\newblock Stability and chaos in dynamical last passage percolation.
\newblock {\em Commun. Am. Math. Soc.}, 4:387--479, 2024.

\bibitem{garpetsch18}
C.~Garban, G.~Pete, and O.~Schramm.
\newblock The scaling limits of near-critical and dynamical percolation.
\newblock {\em J. Eur. Math. Soc. (JEMS)}, 20(5):1195--1268, 2018.

\bibitem{garste15}
C.~Garban and J.~E. Steif.
\newblock {\em Noise sensitivity of {B}oolean functions and percolation}.
\newblock Institute of Mathematical Statistics Textbooks. Cambridge University
  Press, New York, 2015.

\bibitem{hamwel65}
J.~M. Hammersley and D.~J.~A. Welsh.
\newblock First-passage percolation, subadditive processes, stochastic
  networks, and generalized renewal theory.
\newblock In {\em Proc. {I}nternat. {R}es. {S}emin., {S}tatist. {L}ab., {U}niv.
  {C}alifornia, {B}erkeley, {C}alif}, pages 61--110. Springer-Verlag, New York,
  1965.

\bibitem{odonnell14}
R.~O'Donnell.
\newblock {\em Analysis of {B}oolean functions}.
\newblock Cambridge University Press, New York, 2014.

\bibitem{tasvan23}
V.~Tassion and H.~Vanneuville.
\newblock Noise sensitivity of percolation via differential inequalities.
\newblock {\em Proc. Lond. Math. Soc. (3)}, 126(4):1063--1091, 2023.

\end{thebibliography}

\newcommand{\noopsort}[1]{}\def\cprime{$'$}

\end{document}